\documentclass[12pt,reqno]{amsart} 

\usepackage{fullpage,colonequals,enumerate,rotate,url}
\usepackage{amsmath,amssymb,amscd,latexsym,amsthm,stmaryrd}
\usepackage{amsfonts, mathrsfs}
\include{xy}
\include{diagxy} 
\xyoption{curve}
\usepackage[all,cmtip]{xy}

\usepackage{color}

\newtheorem{theorem}{Theorem}[section]
\newtheorem{lemma}[theorem]{Lemma}
\newtheorem{corollary}[theorem]{Corollary}
\newtheorem{proposition}[theorem]{Proposition}

\theoremstyle{definition}
\newtheorem{definition}[theorem]{Definition}

\newtheorem{example}[theorem]{Example}

\theoremstyle{remark}
\newtheorem{remark}[theorem]{Remark}

\newcounter{tempenum}

\newcommand{\Q}{\mathbb{Q}}

\newcommand{\ra}{\rightarrow}

\newcommand{\Z}{\mathbb{Z}}
\newcommand{\F}{\mathbb{F}}

\newcommand{\tors}{\operatorname{tors}}

\newcommand{\AV}{\textup{\textsf{AV}}}
\newcommand{\pp}{\mathfrak{p}}
\newcommand{\qq}{\mathfrak{q}}
\newcommand{\GL}{{\rm GL}}
\newcommand{\CC}{\mathbb{C}}
\newcommand{\PP}{\mathbb{P}}
\newcommand{\Fbar}{\overline{\F}}
\newcommand{\kbar}{\overline{k}}
\newcommand{\Qbar}{\overline{\Q}}
\newcommand{\calC}{\mathcal{C}}
\newcommand{\calD}{\mathcal{D}}
\newcommand{\calO}{\mathcal{O}}
\newcommand{\OO}{\mathscr{O}}

\newcommand{\directsum}{\oplus}
\newcommand{\Directsum}{\bigoplus}
\newcommand{\Galois}{\mathcal{G}}
\newcommand{\isom}{\simeq}
\newcommand{\longisomto}{\stackrel{\sim}{\To}}
\newcommand{\intersect}{\cap}
\newcommand{\Intersection}{\bigcap}
\newcommand{\injects}{\hookrightarrow}
\newcommand{\surjects}{\twoheadrightarrow}
\newcommand{\tensor}{\otimes}
\newcommand{\boldalpha}{\boldsymbol{\alpha}}
\newcommand{\To}{\longrightarrow}

\DeclareMathOperator{\Aut}{Aut}
\DeclareMathOperator{\Char}{char}
\DeclareMathOperator{\coker}{coker}

\DeclareMathOperator{\End}{End}
\DeclareMathOperator{\Ext}{Ext}
\DeclareMathOperator{\Frac}{Frac}

\DeclareMathOperator{\Gal}{Gal}
\DeclareMathOperator{\Hom}{Hom}
\DeclareMathOperator{\HOM}{\mathscr{H}\!\!\mathit{om}}

\DeclareMathOperator{\Lie}{Lie}
\DeclareMathOperator{\M}{M}

\DeclareMathOperator{\opp}{opp}
\DeclareMathOperator{\Pic}{Pic}

\DeclareMathOperator{\rk}{rk}
\DeclareMathOperator{\Sets}{\textbf{Sets}}

\usepackage{color}

\usepackage{microtype}

\usepackage[
	pdfauthor={Bjorn Poonen}, 
	pdftitle={Abelian varieties isogenous to a power of an elliptic curve},
]{hyperref}

\begin{document}

\title[Abelian varieties isogenous to a power]{Abelian varieties isogenous to a power of an elliptic curve}

\author{Bruce~W.~Jordan}
\address{Department of Mathematics, Baruch College, The City University
of New York, One Bernard Baruch Way, New York, NY 10010-5526, USA}
\email{bruce.jordan@baruch.cuny.edu}

\author{Allan~G.~Keeton}
\address{Center for Communications Research, 805 Bunn Drive, Princeton, 
NJ 08540-1966, USA}
\email{agk@idaccr.org}

\author{Bjorn~Poonen}
\address{Department of Mathematics, Massachusetts Institute of Technology, Cambridge, MA 02139-4307, USA}
\email{poonen@math.mit.edu}
\urladdr{\url{http://math.mit.edu/~poonen/}}

\author{Eric~M.~Rains}
\address{Department of Mathematics, California Institute of Technology,
Pasadena, CA 91125, USA}
\email{rains@caltech.edu}

\author{Nicholas~Shepherd-Barron}
\address{Mathematics Department, King's College London,
Strand, London WC2R 2LS, United Kingdom}
\email{N.I.Shepherd-Barron@kcl.ac.uk}

\author{John~T.~Tate}
\address{Mathematics Department, Harvard University, 1 Oxford Street,
Cambridge MA 02138-2901, USA}
\email{Tate@math.utexas.edu}

\thanks{B.P.\ was supported in part by National Science Foundation grant DMS-1069236 and DMS-1601946 and grants from the Simons Foundation (\#340694 and \#402472 to Bjorn Poonen).  Any opinions, findings, and conclusions or recommendations expressed in this material are those of the authors and do not necessarily reflect the views of the National Science Foundation or the Simons Foundation.}

\date{September 20, 2016}  

\begin{abstract} 
Let $E$ be an elliptic curve over a field $k$.
Let $R\colonequals \End E$.
There is a functor $\HOM_R(-,E)$ from
the category of finitely presented torsion-free left $R$-modules
to the category of abelian varieties isogenous to a power of $E$,
and a functor $\Hom(-,E)$ in the opposite direction.
We prove necessary and sufficient conditions on $E$
for these functors to be equivalences of categories.
We also prove a partial generalization in which $E$ is replaced
by a suitable higher-dimensional abelian variety over $\F_p$.
\end{abstract}

\maketitle

\section{Introduction}
\label{S:introduction}

 Let $E$ be an elliptic curve over a field $k$.
Let $R\colonequals \End E$.
We would like to classify all abelian varieties isogenous to a power of $E$.
There is a functor $\HOM_R(-,E)$ that takes as input 
a finitely presented (f.p.)\ left $R$-module $M$ 
and produces a commutative group scheme.
(This functor appears in articles by Giraud~\cite[\S1]{Giraud1968}
and Waterhouse~\cite[Appendix]{Waterhouse1969},
and is attributed by the former to Serre and Tate;
we will give a self-contained exposition 
in Section~\ref{S:functor to an abelian category}.)
We will prove that when restricted to torsion-free modules,
it becomes a fully faithful functor of additive categories
\begin{align}
\label{E:abelian variety functor}
	\HOM_R(-,E) \colon & \{ \textup{f.p.\ torsion-free left $R$-modules} \}^{\opp} \\
\nonumber	& \rightarrow
	\{\textup{abelian varieties isogenous to a power of $E$}\}.
\end{align}
In the other direction, we have a functor
\begin{align}
\label{E:reverse abelian variety functor}
	\Hom(-,E) \colon & \{\textup{abelian varieties isogenous to a power of $E$}\} \\
\nonumber	& \rightarrow
	\{ \textup{f.p.\ torsion-free left $R$-modules} \}^{\opp} 
\end{align}
that provides the inverse on the essential image 
of~\eqref{E:abelian variety functor}.
These are useful because the modules can be classified for each possible $R$.

 We find necessary and sufficient conditions on $E$ for 
\eqref{E:abelian variety functor} and~\eqref{E:reverse abelian variety functor}
to be equivalences of categories.
For simplicity, in this introduction 
we state the answer only for elliptic curves over finite fields.

\begin{theorem}
\label{T:elliptic curve over finite field}
 Let $E$ be an elliptic curve over a finite field $k = \F_q$.
Let $R\colonequals \End E$.
Let $\pi \in R$ be the $q$-power Frobenius endomorphism.
Then~\eqref{E:abelian variety functor} and~\eqref{E:reverse abelian variety functor}
are equivalences of categories
if and only if one of the following holds\textup{:}
\begin{itemize}
\item $E$ is ordinary and $\Z[\pi]=R$\textup{;}
\item $E$ is supersingular, $k=\F_p$, and $\Z[\pi]=R$\textup{;} or
\item $E$ is supersingular, $k=\F_{p^2}$, and $R$ is of rank~$4$ over $\Z$.
\end{itemize}
\end{theorem}

\noindent
Theorem~\ref{T:elliptic curve over finite field} is close
to many results in the literature.
Waterhouse in~\cite{Waterhouse1969} proves
many results relating the isogeny class of an elliptic curve $E$ 
to the ideal classes of $\End E$,
and he also considers such issues when $E$ is replaced by an abelian variety.
An analogue of Theorem~\ref{T:elliptic curve over finite field} 
with the functors $\HOM$ and $\Hom(-,E)$
replaced by similarly-defined functors $\tensor$ and $\Hom(E,-)$
is proved in Serre's appendix to~\cite{Lauter2002}
in the case where $\Z[\pi]$ is 
the maximal order in an imaginary quadratic field
(in this case, $R = \Z[\pi]$ necessarily).
Other cases are handled in 
\cite{Shioda-Mitani}, \cite{Lange1975}, \cite{Schoen1992}, 
and especially Kani's work~\cite{Kani2011};
although these works do not define the functor $\HOM$,
they too classify all abelian varieties isogenous to a power of $E$ 
in the case where $E$ is ordinary and $\rk R = 2$ 
(see Theorems 1, 2, and~3 of~\cite{Kani2011}).
In fact, at one point (in the proof of 
our Theorem~\ref{T:maximal order equivalence}\eqref{I:fully faithful}), 
we make use of one of the easier results of~\cite{Kani2011}.

The category of all ordinary abelian varieties over a finite field
is equivalent to the category of \emph{Deligne modules}~\cite{Deligne}, 
which are f.p.\ torsion-free $\Z$-modules provided with an 
endomorphism that corresponds to the Frobenius.
The ordinary case of Theorem~\ref{T:elliptic curve over finite field}
could be deduced from Deligne's equivalence.
For a \emph{prime} ground field $\F_p$,
Centeleghe and Stix~\cite{Centeleghe-Stix2015} 
extended Deligne's equivalence 
to a category including most non-ordinary abelian varieties.
For suitable abelian varieties $B$ over $\F_p$,
this leads to a classification of the quotients of powers of $B$;
in particular, when $B$ is simple, these quotients 
are the abelian varieties isogenous to a power of $B$.
Centeleghe and Stix did not mention the functor $\HOM_R(-,B)$,
but in Section~\ref{S:CS} we prove that 
a functor they used is isomorphic to $\HOM_R(-,B)$.
Combining their work with ours, we can rewrite their classification
in terms of the functor $\HOM_R(-,B)$.
In particular, this yields a second proof of 
Theorem~\ref{T:elliptic curve over finite field}
in the case where the ground field $k$ is $\F_p$.
Our first proof, although only for elliptic curves,
applies also to non-ordinary elliptic curves over $\F_{p^n}$ for $n>1$
and to elliptic curves over infinite fields 
(see Theorems \ref{T:End E = Z} and~\ref{T:ordinary over infinite field},
for example).
It includes the quaternionic
endomorphism case, and also determines exactly when
the functors above give an equivalence.

Let us now outline the rest of the paper.
Section~\ref{S:notation} introduces notation to be used.
If $R$ is the endomorphism ring of an elliptic curve,
then $R$ is $\Z$, an imaginary quadratic order, or a maximal quaternionic order;
Section~\ref{S:classifying torsion-free modules} reviews the classification
of f.p.\ torsion-free left $R$-modules in each case,
and in a little more generality.
Section~\ref{S:categorical constructions}
introduces the two functors above and proves their basic properties;
in particular it is shown that applying $\HOM_R(-,E)$ to torsion-free modules
produces abelian varieties isogenous to a power of $E$.
Moreover, Section~\ref{S:duality} relates duality of modules to duality of abelian varieties.
Section~\ref{S:new proof} proves that 
when $E$ is a supersingular elliptic curve over $\F_{p^2}$
with $\rk \End E = 4$, the functors 
\eqref{E:abelian variety functor} and~\eqref{E:reverse abelian variety functor}
are equivalences of categories,
so that there is a clean classification of abelian varieties isogenous to
a power of $E$.
In preparation for the other cases, Section~\ref{S:kernel subgroups}
defines the notion of a kernel subgroup,
and shows that the functors
\eqref{E:abelian variety functor} and~\eqref{E:reverse abelian variety functor}
are equivalences of categories
if and only if every finite subgroup scheme of every power of $E$
is a kernel subgroup.
All this is combined in Section~\ref{S:abelian varieties},
which gives a complete answer to the question of when
\eqref{E:abelian variety functor} and~\eqref{E:reverse abelian variety functor}
are equivalences of categories.
Section~\ref{S:CS} contains the argument involving
the work of Centeleghe and Stix
for certain abelian varieties of higher dimension over $\F_p$.

\section{Notation}
\label{S:notation}

 Let $R$ be a noetherian integral domain.
Let $K=\Frac R$.
The \emph{torsion submodule} of an $R$-module $M$ is 
\[
	M_{\tors}\colonequals \{m \in M: rm=0 \textup{ for some nonzero $r \in R$}\}.
\]
Call $M$ \emph{torsion-free} if $M_{\tors}=0$.
Call a submodule $N$ of $M$ (or an injection $N \rightarrow M$)
\emph{saturated} if the cokernel of $N \rightarrow M$ is torsion-free.
Given a f.p.\ $R$-module $M$, define its \emph{rank}
as $\rk M \colonequals \dim_K(K \tensor_R M)$.
The notion of rank extends to f.p.\ left modules
over a subring $R$ in a division algebra $K$.

 If $k$ is a field, let $\kbar$ be an algebraic closure of $k$,
let $k_s$ be the separable closure of $k$ in $\kbar$,
and let $\Galois_k\colonequals \Gal(k_s/k)$.
If $G$ is a finite group scheme over a field $k$,
its \emph{order} is $\#G \colonequals \dim_k \Gamma(G,\OO_G)$.
If $A$ is any commutative group scheme over a field $k$,
and $n \in \Z_{>0}$,
then $A[n]$ denotes the group scheme kernel of $A \stackrel{n}\rightarrow A$.
If $\ell$ is a prime not equal to $\Char k$,
then the $\ell$-adic \emph{Tate module} of $A$ is 
\[
	T_\ell A \colonequals \varprojlim_{e} A[\ell^e](k_s).
\]

 If $X$ is a scheme over a field $k$ of characteristic~$p>0$,
and $q$ is a power of $p$,
let $\pi_{X,q} \colon X \rightarrow X^{(q)}$ be the $q$-power Frobenius morphism;
if $k=\F_q$, then let $\pi_X$ be $\pi_{X,q} \colon X \rightarrow X$.
If $E$ is a commutative group scheme over a field $k$,
then $\End E$ denotes its endomorphism ring as a commutative group scheme over $k$,
i.e., the ring of endomorphisms defined over $k$; the same comment applies to $\Hom$.

 Recall that the \emph{essential image} of a functor $F \colon \calC \rightarrow \calD$
consists of the objects of $\calD$ isomorphic to $FC$ for some $C \in \calC$;
from now on, we call this simply the \emph{image} of $F$.

\section{Classifying torsion-free modules}
\label{S:classifying torsion-free modules}

\subsection{Dedekind domains}
\label{S:Dedekind domains}

Suppose that $R$ is a Dedekind domain.  Finitely presented
(henceforth denoted f.p.) torsion-free $R$-modules can be completely classified,
as is well known~\cite[Theorem~4.13]{Reiner2003}.
To describe the result, we need the notion of determinant of a module.
Given a torsion-free $R$-module $M$ of rank~$r$,
its \emph{determinant} $\det M \colonequals \bigwedge^r M$ 
is a f.p.\ torsion-free $R$-module of rank~$1$;
sometimes we identify $\det M$ with its class in $\Pic R$.
For example, if $M = I_1 \directsum \cdots \directsum I_r$,
where each $I_j$ is a nonzero ideal of $R$,
then $\rk M = r$ and 
\[
	\det M \; 
	\isom \; I_1 \underset{R}\tensor \cdots \underset{R}\tensor I_r \;
	\isom \; I_1 \cdots I_r \quad \textup{(the product ideal in $R$).}
\]

\begin{theorem}
\label{T:modules over Dedekind domain}
\hfill
\begin{enumerate}[\upshape (a)]
\item\label{I:torsion-free equals projective over Dedekind domain} 
 A f.p.\ $R$-module is torsion-free if and only if it is projective.
\item\label{I:direct sum of ideals over Dedekind domain} 
 Every f.p.\ projective $R$-module is isomorphic to a finite direct sum of invertible ideals.
\item  The isomorphism type of a f.p.\ projective $R$-module 
is determined by its rank and determinant.
\item  Every pair $(r,c) \in \Z_{>0} \times \Pic R$ arises as the rank and determinant 
of a nonzero f.p.\ projective $R$-module $M$\textup{;}
one representative is $M \colonequals R^{r-1} \directsum I$ where $[I]=c$.
\end{enumerate}
\end{theorem}

\subsection{Quadratic orders}
\label{S:quadratic orders}

 For a general order in a Dedekind domain, 
the structure theory of torsion-free f.p.\ modules is wild.
Fortunately, for \emph{quadratic} orders there is a theory 
that is only slightly more complicated than that for Dedekind domains.
Recall that if $R_{\max}$ is the ring of integers in a quadratic field $K$,
then every order in $K$ is of the form $R_f\colonequals \Z + f R_{\max}$
for a positive integer $f$ called the \emph{conductor}.
The orders containing $R_f$ are the orders $R_g$ for $g|f$.

\begin{theorem}
\label{T:modules over quadratic order}
 Let $R$ be a quadratic order, i.e., an order in a degree~$2$ extension $K$ of $\Q$.
Let $M$ be a f.p.\ torsion-free $R$-module.
\begin{enumerate}[\upshape (i)]
\item \label{I:chain of quadratic orders} 
There exists a unique chain of orders $R_1 \subseteq \cdots \subseteq R_n$ between $R$ and $K$
and invertible ideals $I_1,\ldots,I_n$ of $R_1,\ldots,R_n$, respectively,
such that $M \isom I_1 \directsum \cdots \directsum I_n$ as an $R$-module.
\item \label{I:quadratic Steinitz class}
The $I_i$ are not unique, but their product $I_1 \cdots I_n$
is an invertible $R_n$-ideal whose class $[M] \in \Pic R_n$ depends only on $M$.
\item \label{I:invariants of quadratic module}
The isomorphism type of $M$ 
is uniquely determined by the chain $R_1 \subseteq \cdots \subseteq R_n$
and the class $[M] \in \Pic R_n$.
\end{enumerate}
\end{theorem}

\begin{proof}
 See~\cite{Borevich-Faddeev1960}.
For generalizations to other integral domains, 
see \cite[Section~7]{Bass1963}, \cite{Borevich-Faddeev1965}, \cite{Levy1985}, 
and the survey article \cite{Salce2002}.
\end{proof}

\subsection{Maximal orders in quaternion algebras}
\label{S:maximal orders}

 Let $B$ be a quaternion division algebra over $\Q$.
Let $\calO$ be a maximal order in $B$.
Given a f.p.\ left $\calO$-module $M$,
the nonnegative integer $\rk M$ 
is the dimension of the left $B$-vector space $B \tensor_\calO M$,
which is also $\frac14 \rk_\Z M$.
Call $M$ \emph{torsion-free} if the natural map $M \rightarrow B \tensor_\calO M$
is an injection, or equivalently if $M$ is torsion-free as a $\Z$-module.

 The classification of f.p.\ torsion-free left $\calO$-modules
is similar to the classification over a Dedekind domain,
and even simpler in ranks at least $2$.

\begin{theorem}
\label{T:modules over quaternion order}
\hfill
\begin{enumerate}[\upshape (a)]
\item\label{I:torsion-free equals projective over O} 
 A f.p.\ left $\calO$-module is torsion-free if and only if it is projective.
\item\label{I:direct sum of ideals over O}
 Every f.p.\ projective left $\calO$-module is isomorphic to a finite direct sum 
of ideals.
\item \label{I:free over quaternion order}
 A f.p.\ projective left $\calO$-module of rank at least $2$ is free.
\end{enumerate}
\end{theorem}

\begin{proof}
\hfill
\begin{enumerate}[\upshape (a)]
\item  See~\cite[Corollary~21.5]{Reiner2003}.
\item  This follows from the final statement of~\cite[Corollary~21.5]{Reiner2003}.
\item  This is a classical result due to Eichler~\cite{Eichler1938}; 
see also~\cite[Theorem~3.5]{Shioda1979}.
\qedhere
\end{enumerate}
\end{proof}

\section{Categorical constructions}
\label{S:categorical constructions}

\subsection{A functor to an abelian category}
\label{S:functor to an abelian category}

We recall the following general construction 
(cf.\ \cite[\S1]{Giraud1968}, \cite[Appendix]{Waterhouse1969}, 
or~\cite[pp.~50--51]{Serre1985}).
Fix an abelian category $\calC$, an object $E \in \calC$,
a ring $R$,
and a ring homomorphism $R \rightarrow \End E$.
For each f.p.\ left $R$-module $M$,
choose a presentation 
\begin{equation}
\label{E:presentation of M}
	R^m \to R^n \to M \to 0.
\end{equation}
If we view $R^m$ and $R^n$ as spaces of row vectors,
then the $R$-module homomorphism $R^m \rightarrow R^n$ 
is represented by right-multiplication by some matrix $X \in \M_{m,n}(R)$.
Since $R$ acts on $E$, left-multiplication by $X$ defines a morphism
$E^n \rightarrow E^m$, whose kernel we call $A$:
\begin{equation}
\label{E:definition of A}
	0 \to A \to E^n \to E^m.
\end{equation}
For any $C \in \calC$, applying $\Hom(C,-)$ yields an exact sequence
\[
	0 \to \Hom(C,A) \to \Hom(C,E)^n \to \Hom(C,E)^m.
\]
On the other hand, applying $\Hom_R(-,\Hom(C,E))$ 
to~\eqref{E:presentation of M} yields an exact sequence
\[
	0 \to \Hom_R(M,\Hom(C,E)) \to \Hom(C,E)^n \to \Hom(C,E)^m.
\]
Comparing the previous two sequences yields an isomorphism
\[
	\Hom(C,A) \isom \Hom_R(M,\Hom(C,E)),
\]
and it is functorial in $C$.
This gives a presentation-independent description of $A$ 
up to isomorphism as an object of $\calC$ representing 
the functor $\Hom_R(M,\Hom(-,E)) \colon \calC \rightarrow \Sets$.
Define $\HOM_R(M,E) \colonequals A$.

An $R$-module homomorphism $M \rightarrow M'$ induces a homomorphism
\[
	\Hom_R(M',\Hom(C,E)) \to \Hom_R(M,\Hom(C,E))
\]
for each $C \in \calC$, functorially in $C$, so by Yoneda's lemma
it induces also a morphism between the representing objects 
$\HOM_R(M',E) \rightarrow \HOM_R(M,E)$.
Thus we obtain a functor
\begin{equation}
\label{E:functor}
	\HOM_R(-,E) \colon \{ \textup{f.p.\ left $R$-modules} \}^{\opp}
	\rightarrow
	\calC.
\end{equation}

If $0 \rightarrow M_1 \rightarrow M_2 \rightarrow M_3$ 
is an exact sequence of f.p.\ left $R$-modules,
then for each $C \in \calC$,
\[
	0 \to \Hom_R(M_1,\Hom(C,E)) 
	\to \Hom_R(M_2,\Hom(C,E)) 
	\to \Hom_R(M_3,\Hom(C,E))
\]
is exact.
This implies that the sequence of representing objects
\[
	0 \to \HOM_R(M_1,E) \to \HOM_R(M_2,E) \to \HOM_R(M_3,E)
\]
is exact.
That is, the functor $\HOM_R(-,E)$ is left exact.

\begin{remark}
\label{R:tensor functor}
Following Serre's appendix to~\cite{Lauter2002},
one can also define a functor 
\[
	- {\tensor_R} E \colon \{ \textup{f.p.\ right $R$-modules} \}
	\rightarrow
	\calC.
\]
Namely, given a f.p.\ right $R$-module $M$, choose a presentation
\[
	R^m \to R^n \to M \to 0,
\]
and define $M \tensor_R E$ as the cokernel of $E^m \to E^n$.
\end{remark}

\subsection{The functor for an elliptic curve produces abelian varieties}

The category of commutative proper group schemes over a field $k$
is an abelian category 
(the hardest part of this statement is the existence of cokernels, 
which is~\cite[Corollaire~7.4]{Grothendieck1961}).
{}From now on, we assume that $\calC$ is this category.

\begin{proposition}
\label{P:functor of points}
Let $M$ be an $R$-module.
Let $A \colonequals \HOM_R(M,E)$.
For every $k$-algebra $L$, we have $A(L) \isom \Hom_R(M,E(L))$.
\end{proposition}

\begin{proof}
Taking $L$-points of \eqref{E:definition of A} yields an exact sequence
\[
	0 \to A(L) \to E(L)^n \to E(L)^m.
\]
On the other hand, applying $\Hom_R(-,E(L))$ to \eqref{E:presentation of M}
yields an exact sequence
\[
	0 \to \Hom_R(M,E(L)) \to E(L)^n \to E(L)^m.
\]
The maps $E(L)^n \to E(L)^m$ in both sequences are the same, 
so the result follows.
\end{proof}

\begin{proposition}
\label{P:dimension of A}
 Let $E$ be an abelian variety over a field $k$.
Let $R$ be a domain that is f.p.\ as a $\Z$-module.
Let $R \rightarrow \End E$ be a ring homomorphism.
Let $M$ be a f.p.\ left $R$-module.
Let $A \colonequals \HOM_R(M,E)$.
Then $\dim A = (\rk M)(\dim E)$.
\end{proposition}

\begin{proof}
For any $n \ge 1$, 
the presentation $R \stackrel{n}\rightarrow R \rightarrow R/nR \rightarrow 0$
shows that $\HOM_R(R/nR,E) \isom E[n]$.
If $M$ is torsion, then it is a quotient of $(R/nR)^m$ for some $m,n \ge 1$;
then $A \subseteq E[n]^m$, so $A$ is finite.

In general, let $r=\rk M$.
There is an exact sequence 
\[
	0 \rightarrow R^r \rightarrow M \rightarrow T \rightarrow 0
\]
for some torsion module $T$;
this yields 
\begin{equation}
\label{E:A to E^r}
	0 \rightarrow \HOM_R(T,E) \rightarrow A \rightarrow E^r.
\end{equation}
By the previous paragraph, $\HOM_R(T,E)$ is finite, so $\dim A \le r \dim E$.
There exists a nonzero $\rho \in R$ such that $\rho T = 0$.
Since $R$ is f.p.\ as a $\Z$-module, 
it follows that there exists a positive integer $n$ such that $nT=0$.
Then $R^r \stackrel{n}\rightarrow R^r$ factors as $R^r \injects M \rightarrow R^r$,
which induces $E^r \rightarrow A \rightarrow E^r$ whose composition is multiplication by $n$,
which is surjective.
Thus $A \rightarrow E^r$ is surjective, so $\dim A \ge r \dim E$.
Hence $\dim A = r \dim E$.
\end{proof}

 If $E$ is an elliptic curve, and $I$ is a subset of $\End E$,
let $E[I] \colonequals \Intersection_{\alpha \in I} \ker \alpha$.

\begin{theorem}
\label{T:functor for E}
 Let $E$ be an elliptic curve over a field $k$.
Let $R$ be a saturated subring of $\End E$ \textup{(}saturated as a $\Z$-module\textup{)}.
Let $M$ be a torsion-free f.p.\ left $R$-module.
Let $A \colonequals \HOM_R(M,E)$.
Then 
\begin{enumerate}[\upshape (a)]
\item\label{I:A is an abelian variety}
The group scheme $A$ is an abelian variety isogenous to a power of~$E$.
\item\label{I:Hom is exact}
The functor $\HOM_R(-,E)$ is exact.
\item\label{I:image of E^r --> E^s}
If $f \colon E^r \rightarrow E^s$ is a homomorphism arising from 
applying $\HOM_R(-,E)$ to an $R$-homomorphism $g \colon R^s \rightarrow R^r$,
then the image of $f$ is 
isomorphic to $\HOM_R(N,E)$ for some f.p.\ torsion-free $R$-module $N
\subseteq R^r$. 
(Moreover, if $R = \End E$,
then every homomorphism $f \colon E^r \rightarrow E^s$ 
arises from some $g$.)
\item \label{I:Hom for ideal}
If $I$ is a nonzero left $R$-ideal,
then $\HOM_R(R/I,E) \isom E[I]$ and $\HOM_R(I,E) \isom E/E[I]$.
\item \label{I:order of Hom(N,E)}
If $T$ is an $R$-module that is finite as a set, 
then $\HOM_R(T,E)$ is a finite group scheme of order $(\#T)^{2/\rk R}$.
\item \label{I:n-torsion of HOM}
If $n \in \Z_{>0}$, then $A[n] \isom \HOM_R(M,E[n])$,
where the latter is defined by using the induced ring homomorphism $R \rightarrow \End E[n]$.
\item \label{I:Tate module of HOM}
If $\ell$ is a prime not equal to $\Char k$,
then $T_\ell A \isom \Hom_R(M,T_\ell E)$.
\end{enumerate}
\end{theorem}

\begin{proof}
\hfill
\begin{enumerate}[\upshape (a)]
\item 
 Let $r=\rk M = \dim A$.
The proof of Proposition~\ref{P:dimension of A} shows that $A$ admits a surjection to $E^r$
with finite kernel, so if $A$ is an abelian variety, it is isogenous to $E^r$.

 The ring $R$ is either $\Z$, a quadratic order, or a maximal quaternionic order.
In the first and third cases, $M$ is projective of rank $r$ over $R$ 
(the quaternionic case is 
Theorem~\ref{T:modules over quaternion order}\eqref{I:torsion-free equals projective over O});
in other words, $M$ is a direct summand of $R^n$ for some $n$;
thus $A$ is a direct factor of $E^n$, so $A$ is an abelian variety.

 So suppose that $R$ is a quadratic order.
Let $c$ be the conductor, i.e., the index of $R$ in its integral closure.
Let $\ell$ denote a prime.
If $\ell \nmid c$, 
then the semi-local ring $R \tensor \Z_{(\ell)}$ is a Dedekind domain,
but a semi-local Dedekind domain is a principal ideal domain,
so $M \tensor \Z_{(\ell)}$ is free of rank~$r$ over $R \tensor \Z_{(\ell)}$,
and $M/\ell M$ is free of rank $r$ over $R/\ell R$.

 We claim that $A$ is smooth.
This is automatic if $\Char k=0$.
So suppose that $\Char k=p>0$.
By \cite[Theorem~4.2]{Waterhouse1969}, 
we have $p \nmid c$,
so by the above, $M/pM$ is free of rank $r$ over $R/pR$.
By Proposition~\ref{P:functor of points}, applying $\Hom_R(M,-)$ to 
\[
	0 \rightarrow \Lie E \rightarrow E(k[\epsilon]/(\epsilon^2)) \rightarrow E(k) \rightarrow 0
\]
yields
\[
	0 \rightarrow \Hom_R(M,\Lie E) \rightarrow A(k[\epsilon]/(\epsilon^2)) \rightarrow A(k) \rightarrow 0.
\]
Thus
\[
	\Lie A \isom \Hom_R(M,\Lie E) \isom \Hom_{R/pR}(M/pM,\Lie E) \isom (\Lie E)^r.
\]
In particular, $\dim \Lie A = r$, so $A$ is smooth.

Since $A$ is also proper,
it is an extension of a finite \'etale commutative group scheme
$\Phi$ by an abelian variety $B$.
The constructed surjection $A \rightarrow E^r$ with finite kernel
restricts to a homomorphism $B \rightarrow E^r$ with finite kernel,
and it must still be surjective since $E^r$ does not have algebraic subgroups
of finite index;
thus $B$ is isogenous to $E^r$.
Since $B(\kbar)$ is divisible, the extension splits over $\kbar$.
In particular, for each prime $\ell$, 
\begin{equation}
\label{E:l-torsion}
	\#A(\kbar)[\ell] = \#E(\kbar)[\ell]^r \#\Phi[\ell].
\end{equation}
On the other hand, Proposition~\ref{P:functor of points} implies
\begin{equation}
\label{E:l-torsion group}
	A(\kbar)[\ell] = \Hom_R(M,E(\kbar)[\ell]) 
	= \Hom_{R/\ell R}(M/\ell M,E(\kbar)[\ell]).
\end{equation}
We claim that
\begin{equation}
\label{E:l-torsion again}
	\#A(\kbar)[\ell] = \#E(\kbar)[\ell]^r.
\end{equation}
If $\ell \nmid c$, then $M/\ell M$ is free of rank~$r$ over $R/\ell R$,
so \eqref{E:l-torsion again} holds;
in particular, this holds if $\ell=\Char k$.
Now suppose that $\ell|c$.
Then $R/\ell R \isom \F_\ell[e]/(e^2)$.
Every module over $R/\ell R$ is a direct sum of copies of $\F_\ell$ and $\F_\ell[e]/(e^2)$.
Since $R$ is saturated in $\End E$, the homomorphisms
\[
	\frac{R}{\ell R} \rightarrow \frac{\End E}{\ell(\End E)} \rightarrow \End E(\kbar)[\ell]
\]
are injective, but $\#E(\kbar)[\ell] = \ell^2 = \#(R/\ell R)$, 
so $E(\kbar)[\ell]$ is free of rank~$1$ over $R/\ell R$.
The equality $\#\Hom_{R/\ell R}(N,R/\ell R) = \#N$ holds for $N = \F_\ell$ and $N = \F_\ell[e]/(e^2)$,
so it holds for every finite $(R/\ell R)$-module $N$,
and in particular for $M/\ell M$.
Thus \eqref{E:l-torsion group} implies 
\[
	\#A(\kbar)[\ell] = \#(M/\ell M) = \#(R/\ell R)^r = \#E(\kbar)[\ell]^r.
\]
Thus \eqref{E:l-torsion again} holds for all $\ell$.

 Comparing \eqref{E:l-torsion} and~\eqref{E:l-torsion again}
shows that $\#\Phi[\ell]=1$ for all $\ell$, so $\Phi$ is trivial.
Thus $A=B$, an abelian variety.
\item \label{itm:b}
By Lemma~\ref{L:projective} below,
it suffices to show that if $M \rightarrow P$ is an injection of modules
with $P$ projective, then $\HOM_R(P,E) \rightarrow \HOM_R(M,E)$ is surjective.
We have an exact sequence
\[
	0 \rightarrow \HOM_R(P/M,E) \rightarrow \HOM_R(P,E) \rightarrow \HOM_R(M,E).
\]
By~\eqref{I:A is an abelian variety},
$\HOM_R(P,E)$ and $\HOM_R(M,E)$ are abelian varieties,
so the image $I$ of $\HOM_R(P,E) \rightarrow \HOM_R(M,E)$ is an abelian subvariety of $\HOM_R(M,E)$.
By Proposition~\ref{P:dimension of A}, 
\[
	\dim \HOM_R(P,E) = \dim \HOM_R(P/M,E) + \dim \HOM_R(M,E),
\]
so $\dim I = \dim \HOM_R(M,E)$.
Thus $I = \HOM_R(M,E)$; i.e., $\HOM_R(P,E) \rightarrow \HOM_R(M,E)$ is surjective.
\item
Since $\HOM_R(-,E)$ is exact, it transforms the co-image of $g$ into the image of $f$.
(Co-image equals image in any abelian category, 
though the proof above does not need this.)
\item 
The proof of \cite[Proposition~A.2]{Waterhouse1969} 
shows that $\HOM_R(R/I,E) \isom E[I]$
(there $R$ is \emph{equal} to $\End E$, but this is not used).
The proof of \cite[Proposition~A.3]{Waterhouse1969} 
shows that $E/E[I]$ is the connected component of $\HOM_R(I,E)$,
but $\HOM_R(I,E)$ is already connected, by~\eqref{I:A is an abelian variety}.
\item 
The function $(\#T)^{2/\rk R}$ of $T$ is multiplicative in short exact sequences.
So is $\#\HOM_R(T,E)$, since $\HOM_R(-,E)$ is exact.
Thus we may reduce to the case in which $T$ is simple,
i.e., $T \isom R/I$ for some maximal ideal $I$.
Then $\HOM_R(T,E) = E[I]$ by~\eqref{I:Hom for ideal}.
We have $I \supseteq \ell R$ for some prime $\ell$.
If $I = \ell R$, then $E[I] = E[\ell]$, which has order $\ell^2 = \#(R/I)^{2/\rk R}$.

 Now suppose that $I \ne \ell R$.
If $R$ has rank~$2$, then $\#(R/I)=\ell$;
if $R$ has rank~$4$, then $\#(R/I)=\ell^2$.
Choose $f \in I \setminus \ell R$;
then $f$ does not kill $E[\ell]$, so $E[I] \subsetneq E[\ell]$.
Thus $\#E[I] \le \ell = \#(R/I)^{2/\rk R}$.
Thus $\#\HOM_R(T,E) \le (\#T)^{2/\rk R}$ holds for each Jordan--H\"older factor of $R/\ell R$,
but for $T=R/\ell R$ equality holds,
so all the inequalities must have been equalities.
\item 
Start with the exact sequence
\[
	0 \rightarrow E[n] \rightarrow E \stackrel{n}\rightarrow E.
\]
Given $S \in \calC$, apply the left exact functors $\Hom_{\calC}(S,-)$
and then $\Hom_R(M,-)$; taken for all $S$, this produces an exact sequence of 
representable functors
\[
	0 \rightarrow \HOM_R(M,E[n]) \rightarrow \HOM_R(M,E) \stackrel{n}\rightarrow \HOM_R(M,E).
\]
Hence $\HOM_R(M,E[n]) \isom A[n]$.
\item 
We have
\begin{align*}
	T_\ell A 
	&\colonequals \varprojlim_e A[\ell^e](k_s) \\
	&\isom \varprojlim_e \HOM_R(M,E[\ell^e])(k_s) \qquad\textup{(by~\eqref{I:n-torsion of HOM})}\\
	&\isom \varprojlim_e \Hom_R(M,E[\ell^e](k_s)) \qquad\textup{(by Proposition~\ref{P:functor of points} with $E$ replaced by $E[\ell^e]$)} \\
	&\isom \Hom_R(M,\varprojlim_e E[\ell^e](k_s)) \\
	&=: \Hom_R(M,T_\ell E). &\qedhere
\end{align*}
\end{enumerate}
\end{proof}

The following was used in the proof of Theorem~\ref{T:functor for E}\eqref{I:Hom is exact}.

\begin{lemma}
\label{L:projective}
 Let $\calC$ be an abelian category with enough projectives.
Let $F \colon \calC^{\opp} \rightarrow \calD$ be a left exact functor.
Suppose that for each monomorphism $M \rightarrow P$ with $P$ projective,
the morphism $FP \rightarrow FM$ is an epimorphism.
Then $F$ is exact.
\end{lemma}

\begin{proof}
 Given $A \in \calC$, choose an epimorphism $P \rightarrow A$ with $P$ projective,
and let $K$ be the kernel.
The sequence $0 \rightarrow K \rightarrow P \rightarrow A \rightarrow 0$ yields
\[
	0 \rightarrow FA \rightarrow FP \rightarrow FK \rightarrow (R^1 F)A \rightarrow (R^1 F)P = 0,
\]
and the hypothesis implies that $FP \rightarrow FK$ is surjective, so $(R^1 F)A=0$.
This holds for all $A$, so $F$ is exact.
\end{proof}

\begin{remark}
The hypothesis that $R$ is saturated in Theorem~\ref{T:functor for E}
cannot be dropped.
For example, if $E$ is an elliptic curve over $\CC$ with $\End E = \Z[i]$,
and $R$ is the subring $\Z[2i]$,
then the $R$-module $\Z[i]$ has a presentation 
\[
	R \stackrel{\begin{pmatrix} 2i \\ -2 \end{pmatrix}}\To R^2 \stackrel{\begin{pmatrix} 1 & i \end{pmatrix}}\To \Z[i] \To 0,
\]
so by definition,
%
\[
	\HOM_R(\Z[i],E) 
	\;\isom\; \ker \raisebox{1ex}{$\left( \raisebox{-1.2ex}{$E^2 \stackrel{\begin{pmatrix} 2i &
            -2 \end{pmatrix}} \To E$} \right)$}
	\;\isom\; E \times E[2],
\]
which is not an abelian variety.
Moreover, applying $\HOM_R(-,E)$ to the injection $\Z[i] \stackrel{2}\rightarrow R$
yields a homomorphism $E \rightarrow E \times E[2]$, which is not surjective,
so $\HOM_R(-,E)$ is not exact.
Finally, $\Z[i]$ is isomorphic as $R$-module to the $R$-ideal $I\colonequals 2\Z[i]$,
so $\HOM_R(I,E)$ is not an abelian variety.
\end{remark}

\subsection{Duality of abelian varieties}
\label{S:duality}

 Let $E$ be an elliptic curve over a field $k$.
Let $R\colonequals \End E$.
The \emph{Rosati involution}, sending an endomorphism to its dual,
is an isomorphism $R \rightarrow R^{\opp}$.
If $M$ is a left $R$-module,
then $M^*\colonequals \Hom_R(M,R)$ 
(the group of homomorphisms of left $R$-modules)
is a \emph{right} $R$-module: given $f \in M^*$ and $r \in R$,
let $f \cdot r$ be the composition $M \stackrel{f}\rightarrow R \stackrel{r}\rightarrow R$.
In other words, $M^*$ is a left $R^{\opp}$-module,
which we may view as a left $R$-module by using the Rosati involution.
Moreover, if $M$ is f.p., then it is finite over $\Z$,
and then so is $M^*$.
Also, $M^*$ is torsion-free.

 Given an abelian variety $A$, let $A^\vee$ be the dual abelian variety.
The following lets us understand the duals of abelian varieties arising from modules.

\begin{theorem}
\label{T:dual abelian variety}
Given a f.p.\ torsion-free left $R$-module $M$,
we have
\[
	\HOM_R(M,E)^\vee \isom \HOM_R(M^*,E),
\]
functorially in $M$.
\end{theorem}

\begin{proof}
 Let $M$ be a f.p.\ torsion-free left $R$-module.
Choose a presentation
\begin{equation}
\label{E:presentation for duality}
	R^n \stackrel{P}\To R^m \To M \To 0,
\end{equation}
where $P \in \M_{m \times n}(R)$.
Apply $\HOM_R(-,E)$ to obtain
\[
	0 \To A \To E^m \stackrel{P^T}\To E^n 
\]
where $P^T$ is the transpose of $P$.
Taking dual abelian varieties yields
\begin{equation}
\label{E:A^vee for duality}
	E^n \stackrel{P^{\dagger}}\To E^m \To A^\vee \To 0,
\end{equation}
where $P^{\dagger}$ is obtained from $P$ by applying the Rosati involution entrywise.

On the other hand, applying $\Hom_R(-,R)$ to~\eqref{E:presentation for duality} yields
\[
	0 \To M^* \To R^m \stackrel{P^{T\dagger}}\To R^n 
\]
and applying $\HOM_R(-,E)$ yields
\[
	E^n \stackrel{P^{\dagger}}\To E^m \To \HOM_R(M^*,E) \To 0.
\]
Comparing with~\eqref{E:A^vee for duality} shows that 
\begin{equation}
\label{E:final duality}
	\HOM_R(M^*,E) \isom A^\vee = \HOM_R(M,E)^\vee.
\end{equation}

 Given a homomorphism of f.p.\ torsion-free left $R$-modules $M \stackrel{f}\rightarrow N$,
we can build a commutative diagram
\[
\xymatrix{
	R^n \ar[r] \ar[d] & R^m \ar[r] \ar[d] & M \ar[r] \ar[d]^f & 0 \\
	R^i \ar[r] & R^j \ar[r] & N \ar[r] & 0 \\
}
\]
and apply the constructions above to show that \eqref{E:final duality}
is functorial in $M$.
\end{proof}

\subsection{The other Hom functor}

Under the assumptions of Theorem~\ref{T:functor for E}\eqref{I:Hom is exact},
we have a functor of additive categories
\begin{align*}
	\HOM_R(-,E) \colon & \{ \textup{f.p.\ torsion-free left $R$-modules} \}^{\opp} \\
	& \rightarrow
	\{\textup{abelian varieties isogenous to a power of $E$}\},
\end{align*}
as promised in the introduction.
{}From now on, $\HOM_R(-,E)$ denotes this functor,
restricted to f.p.\ \emph{torsion-free} left $R$-modules.

 Given an abelian variety $A$ over the same field as $E$,
the abelian group $\Hom(A,E)$ (the group of homomorphisms of abelian varieties) 
is a left \hbox{$(\End E)$-module},
and hence also a left $R$-module, 
and it is f.p.\ because it is f.p.\ over $\Z$ \cite[p.~178, Corollary~1]{Mumford1970}.
In fact, we get a functor in the opposite direction:
\begin{align*}
	\Hom(-,E) \colon & \{\textup{abelian varieties isogenous to a power of $E$}\} \\
	& \rightarrow
	\{ \textup{f.p.\ torsion-free left $R$-modules} \}^{\opp} .
\end{align*}

For which elliptic curves $E$ are $\HOM_R(-,E)$ and $\Hom(-,E)$
inverse equivalences of categories?
If we start with the $R$-module $R$ and apply $\HOM_R(-,E)$ and then $\Hom(-,E)$,
we obtain $\End E$,
so we should have $R \isom \End E$ as $R$-modules;
then the only $R$-module endomorphisms of $\End E$
are given by multiplication by elements of $R$,
but multiplication by elements of $\End E$ also give endomorphisms,
so $R = \End E$.
Thus we assume from now on that $R = \End E$.

\begin{theorem}
\label{T:maximal order equivalence}
 Let $E$ be an elliptic curve over a field.
Let $R\colonequals \End E$.
Then 
\begin{enumerate}[\upshape (a)]
\item\label{I:fully faithful}
The functor $\HOM_R(-,E)$ is fully faithful.
\item\label{I:inverse on essential image}
The functor $\Hom(-,E)$ 
on the image of $\HOM_R(-,E)$
is an inverse to $\HOM_R(-,E)$.
\item \label{I:essential image consists of products}
The image of $\HOM_R(-,E)$
 consists exactly of the products 
of elliptic curves of the form $\HOM_R(I,E)$ for a nonzero left $R$-ideal $I$. 
\end{enumerate}
\end{theorem}

\begin{proof}
\hfill
\begin{enumerate}[\upshape (a)]
\item
 The ring $R$ is $\Z$, a quadratic order, or a maximal quaternionic order.
By Theorem~\ref{T:modules over Dedekind domain}, 
Theorem~\ref{T:modules over quadratic order}\eqref{I:chain of quadratic orders}, 
or Theorem~\ref{T:modules over quaternion order}\eqref{I:direct sum of ideals over O},
respectively,
every f.p.\ torsion-free left $R$-module is a finite direct sum
of nonzero left $R$-ideals.
Thus, \eqref{I:fully faithful} follows if for any two nonzero $R$-ideals $I$ and $J$,
the natural map
\[
	\Hom_R(J,I) \To \Hom(\HOM_R(I,E),\HOM_R(J,E))
\]
is an isomorphism.
If $R=\Z$, this is trivial.
If $R$ is a quadratic order, this is the elliptic curve case of the  
isomorphism given in~(48) in~\cite[Proposition~17]{Kani2011}.
If $R$ is a maximal quaternionic order, then by 
Theorem~\ref{T:modules over quaternion order}\eqref{I:torsion-free equals projective over O} 
all f.p.\ torsion-free left $R$-modules are projective,
i.e., direct summands of f.p.\ free left $R$-modules;
since $\HOM_R(-,E)$ is fully faithful when restricted to free modules,
it is also fully faithful on projective modules.
\item
This is a general property of fully faithful functors.
\item 
As remarked in the proof of~\eqref{I:fully faithful},
every f.p.\ torsion-free left $R$-module is a finite direct sum 
of nonzero left $R$-ideals $I$.
\qedhere
\end{enumerate}
\end{proof}

\section{Maximal abelian varieties over \texorpdfstring{\protect{\boldmath{$\F_{p^2}$}}}{Fp2}}
\label{S:new proof}

 Fix a prime $p$.
Call an abelian variety $A$ over $\F_{p^2}$ \emph{maximal}
if $A$ has the maximum possible number of $\F_{p^2}$-points
for its dimension, namely $(p+1)^{2 \dim A}$.

\begin{proposition}
\label{P:characterization of maximality}
 Let $A$ be a $g$-dimensional abelian variety over $\F_{p^2}$.
Let $\ell$ be a prime not equal to $p$.
The following are equivalent\textup{:}
\begin{enumerate}[\upshape (a)]
\item\label{I:A is maximal} 
The abelian variety $A$ is maximal; i.e., $\# A(\F_{p^2}) = (p+1)^{2g}$.
\item The characteristic polynomial of $\pi_A$ on $T_\ell A$ equals $(x+p)^{2g}$.
\item\label{I:pi_A=p} We have $\pi_A=-p$.
\item We have $A(\F_{p^2}) \isom (\Z/(p+1)\Z)^{2g}$ as abelian groups.
\setcounter{tempenum}{\value{enumi}}
\end{enumerate}
If $E$ is a fixed maximal elliptic curve over $\F_{p^2}$, 
then the following also is equivalent to the above:
\begin{enumerate}[\upshape (a)]
\setcounter{enumi}{\value{tempenum}}
\item\label{I:A and E^g}
The abelian variety $A$ is isogenous to $E^g$.
\end{enumerate}
\end{proposition}

\begin{proof}\hfill
\begin{enumerate}[\upshape (a) \hspace{-0.75 em}]
\item  $\Rightarrow$(b): 
Let $\lambda_1,\ldots,\lambda_{2g} \in \Qbar$ be the eigenvalues of $\pi_A$ acting on $T_\ell A$.
Then $|\lambda_i|=p$ and $\#A(\F_{p^2}) = \prod (1-\lambda_i) = \prod |1-\lambda_i| \le (p+1)^{2g}$;
if equality holds, then $\lambda_i=-p$ for all $i$.
Thus the characteristic polynomial is $(x+p)^{2g}$.

\item  $\Rightarrow$(c):
Since $\pi_A$ is determined by its action on $T_\ell A$,
which is semisimple \cite[pp.~203--206]{Mumford1970}, 
we obtain $\pi_A=-p$.

\item  $\Rightarrow$(d):
We have
\[
	A(\F_{p^2}) = \ker(\pi_A-1) = \ker(-p-1) = A[p+1] \isom (\Z/(p+1)\Z)^{2g}.
\]

\item  $\Rightarrow$(a): Trivial.

\item $\Leftrightarrow$(b): By (a)$\Rightarrow$(b),
the characteristic polynomial of $\pi_E$ is $(x+p)^2$,
so the characteristic polynomial of $\pi_{E^g}$ is $(x+p)^{2g}$.
Two abelian varieties over a finite field are isogenous 
if and only if their characteristic polynomials 
are equal \cite[Theorem~1(c)]{Tate1966}.
\qedhere
\end{enumerate}
\end{proof}

\begin{lemma}
\label{L:homomorphisms defined over Fp^2}
 If $A$ and $B$ are maximal abelian varieties over $\F_{p^2}$,
then any homomorphism $A_{\Fbar_p} \rightarrow B_{\Fbar_p}$
is the base extension of a homomorphism $A \rightarrow B$.
\end{lemma}

\begin{proof}
 Any homomorphism respects the $p^2$-power Frobenius endomorphisms (both are equal to $-p$),
and hence descends to $\F_{p^2}$.
\end{proof}

 Every supersingular elliptic curve over $\Fbar_p$
admits a unique model over $\F_{p^2}$ that is maximal:
the existence is \cite[Lemma~3.21]{Baker-et-al2005},
and uniqueness follows from Lemma~\ref{L:homomorphisms defined over Fp^2}.
In particular, maximal elliptic curves over $\F_{p^2}$ exist.
If $E$ is any such curve, then $E$ is supersingular,
and Lemma~\ref{L:homomorphisms defined over Fp^2} implies
that $\End E = \End E_{\Fbar_p}$, which is a maximal order $\calO$
in a quaternion algebra over $\Q$ ramified at $p$ and $\infty$.
Also, the kernel of the $p$-power Frobenius morphism $E \rightarrow E^{(p)}$
is isomorphic to $\boldalpha_p$.

By Proposition~\ref{P:characterization of maximality}\eqref{I:A is maximal}$\Rightarrow$\eqref{I:A and E^g},
any maximal abelian variety $A$ over $\F_{p^2}$ is isogenous to a power of $E$.
The main result of this section strengthens this as follows:

\begin{theorem}
\label{T:category of maximal abelian varieties}
\hfill
\begin{enumerate}[\upshape (a)]
\item \label{I:maximal is product}
 Every maximal abelian variety $A$ over $\F_{p^2}$
is \emph{isomorphic} to a product of maximal elliptic curves over $\F_{p^2}$.
\item \label{I:equivalence of categories of maximal abelian varieties}
 Fix a maximal elliptic curve $E$ over $\F_{p^2}$.
Let $\calO \colonequals \End E$.
Then the functors $\HOM_{\calO}(-,E)$ and $\Hom(-,E)$
are inverse equivalences of categories.
Also, the categories involved can be rewritten so that $\HOM_R(-,E)$ becomes
\begin{align*}
	\HOM_{\calO}(-,E) \colon & \{ \textup{f.p.\ projective left $\calO$-modules} \}^{\opp} \\
	& \longisomto
	\{ \textup{maximal abelian varieties$/\F_{p^2}$} \}.
\end{align*}
\item \label{I:isomorphic to E^g} 
 Fix a maximal elliptic curve $E$ over $\F_{p^2}$.
Let $g \ge 2$.
Every $g$-dimensional maximal abelian variety over $\F_{p^2}$ is isomorphic to $E^g$.
In particular, any product of $g$ maximal elliptic curves over $\F_{p^2}$
is isomorphic to any other.
\end{enumerate}
The analogous results hold if maximal is replaced by \emph{minimal}\textup{;}
i.e., we consider abelian varieties $A$ over $\F_{p^2}$
such that $\#A(\F_{p^2}) = (p-1)^{2 \dim A}$.
\end{theorem}

 We need a few lemmas for the proof of Theorem~\ref{T:category of maximal abelian varieties}.

\begin{lemma}
\label{L:maximal E over F_p}
There exists an elliptic curve $E$ over $\F_p$ 
such that $E_{\F_{p^2}}$ is maximal.
\end{lemma}

\begin{proof}
There exists an elliptic curve $E$ over $\F_p$ 
with $p+1$ points \cite[Theorem~4.1(5)(i)]{Waterhouse1969}.
The $p$-power Frobenius endomorphism $\pi_E$ of $E$ satisfies $\pi_E^2 = -p$,
so $E_{\F_{p^2}}$ satisfies condition~\eqref{I:pi_A=p} in
Proposition~\ref{P:characterization of maximality}.
\end{proof}

\begin{lemma}
\label{L:separable isogeny}
If $E$ and $E'$ are maximal elliptic curves over $\F_{p^2}$,
there exists a separable isogeny $E \rightarrow E'$.
\end{lemma}

\begin{proof}
For elliptic curves $E$ and $E'$,
write $E \sim E'$ if there exists an isogeny $E \rightarrow E'$
of degree prime to $p$.
The relation $\sim$ is an equivalence relation:
reflexive because of the identity, symmetric because of the dual isogeny
(which has the same degree),
and transitive because of composition of isogenies.

Any isogeny $\phi \colon E \rightarrow E'$ factors as $f \circ \lambda$
where $\deg f = p^n$ for some $n \ge 1$,
and $p \nmid \deg \lambda$.
Here $\lambda$ is separable.
On the other hand, $f$ is a factor of $[p^n]$,
which is purely inseparable if $E$ is maximal.
Thus, assuming that $E$ is maximal, $\phi$ is separable if and only if
$p \nmid \deg \phi$.

Let $E_0$ be the maximal elliptic curve over $\F_{p^2}$
in Lemma~\ref{L:maximal E over F_p}.
Since $\#E_0(\F_{p^2})=\#E(\F_{p^2})$, 
there exists an isogeny $E_0 \rightarrow E$,
which factors as 
$E_0 \stackrel{f}\rightarrow E_0 \stackrel{\lambda}\rightarrow E$,
where $f$ is a power of the $p$-power Frobenius morphism
(which goes from $E_0$ to itself since $E_0$ is definable over $\F_p$),
and $\lambda$ is separable.
By the previous paragraph, $p \nmid \deg \lambda$.
Thus $E_0 \sim E$.
Similarly, $E_0 \sim E'$, so $E \sim E'$.
Thus there exists an isogeny $E \rightarrow E'$ of degree prime to $p$.
Any such isogeny is separable.
\end{proof}

\begin{remark}
Even better, if $E$ and $E'$ are maximal elliptic curves over $\F_{p^2}$,
there exists an isogeny of $\ell$-power degree for any prime $\ell \ne p$:
for an argument due to Serre, see \cite[p.~223]{Mestre1986}.
\end{remark}

\begin{lemma}
\label{L:defined over F_{p^2}}
If $A$ is a maximal abelian variety over $\F_{p^2}$,
then every finite \'etale subgroup scheme of $A_{\Fbar_p}$ is defined over $\F_{p^2}$.
\end{lemma}

\begin{proof}
The $p^2$-power Frobenius field automorphism acts on (prime-to-$p$) 
torsion points of $A_{\Fbar_p}$ as $-p$, 
so it preserves any finite subgroup of order prime to~$p$.
\end{proof}

\begin{lemma}
\label{L:supersingular isogeny factorization}
 Let $A$ be a supersingular abelian variety over a field $k$ of characteristic~$p$.
Every $p$-power order subgroup scheme $G \subseteq A$
is an iterated extension of copies of $\boldalpha_p$.
\end{lemma}

\begin{proof}
 By induction, it suffices to show that if $G \ne 0$,
then $G$ contains a copy of $\boldalpha_p$.
The \emph{$a$-number} $\dim_k \Hom(\boldalpha_p,G)$ is unchanged by field 
extension \cite[Section~1.5]{Li-Oort1998}, 
so we may assume that $k$ is algebraically closed.
Then $A$ is isogenous to a power $E^r$ of a supersingular elliptic curve.
The group scheme $E[p]$ is an extension of $\boldalpha_p$ by $\boldalpha_p$,
so all Jordan--H\"older factors of $E[p^n]$ are isomorphic to $\boldalpha_p$.
The image of $E[p^N]$ under the isogeny $E^r \rightarrow A$ 
contains $A[p^n]$ if $N$ is sufficiently large relative $n$,
and $A[p^n]$ contains $G$ if $n$ is large enough.
Thus all Jordan--H\"older factors of $G$ are isomorphic to $\boldalpha_p$.
\end{proof}

\begin{lemma}
\label{L:action on alpha_p}
 Let $E$ and $E'$ be maximal elliptic curves over $\F_{p^2}$.
Identify $\boldalpha_p$ with a subgroup scheme of each.
Then each homomorphism $E \rightarrow E'$ restricts to a homomorphism $\boldalpha_p \rightarrow \boldalpha_p$
and the resulting map
\begin{equation}
\label{E:Hom to End alpha_p}
	\Hom(E,E') \rightarrow \End \boldalpha_p \isom \F_{p^2}
\end{equation}
is surjective.
\end{lemma}

\begin{proof}
 Since each $\boldalpha_p$ is the kernel of the $p$-power Frobenius morphism,
any homomorphism $E \rightarrow E'$ must map $\boldalpha_p$ to $\boldalpha_p$.
If $E'=E$, then the resulting ring homomorphism 
\[
	\End E \rightarrow \End \boldalpha_p \isom \F_{p^2}
\]
is surjective because every ring homomorphism from $\calO$ to $\F_{p^2}$
is surjective.
In the general case, 
Lemma~\ref{L:separable isogeny}
provides a separable isogeny $\lambda \colon E \rightarrow E'$;
then $\lambda|_{\boldalpha_p} \ne 0$, 
so $\{ \lambda \circ e : e \in \End E \}$ surjects onto $\End \boldalpha_p$.
\end{proof}

\begin{lemma}
\label{L:transitive}
 Let $B$ be a product of $g$ maximal elliptic curves over $\F_{p^2}$.
Then $\Aut B$ acts transitively on the set of subgroup schemes of $B$
isomorphic to $\boldalpha_p$.
Also, if $\ell$ is a prime not equal to $p$,
then $\Aut B$ acts transitively on the set of subgroup schemes of $B$
of order $\ell$.
\end{lemma}

\begin{proof}
 For $\ell \ne p$,
Tate's theorem on homomorphisms~\cite{Tate1966} shows that
\[
	\End B \rightarrow \End(T_\ell B) \isom M_{2g}(\Z_\ell)
\]
is an isomorphism.
In particular, it surjects onto $\End B[\ell] \isom M_{2g}(\F_\ell)$,
and $\Aut B \isom \GL_{2g}(\Z_\ell)$ surjects on $\Aut B[\ell] \isom \GL_{2g}(\F_\ell)$.
For any finite-dimensional vector space $V$, the group $\GL(V)$ acts
transitively on the lines in $V$, so $\Aut B$ acts transitively on the order $\ell$
subgroup schemes of $B$.

 If $F_p \colon B \rightarrow B^{(p)}$ is the $p$-power Frobenius morphism,
then $\ker F_p \isom \boldalpha_p^g$.
Lemma~\ref{L:action on alpha_p} implies that the ring homomorphism
\[
	\End B \rightarrow \End(\ker F_p) = \End(\boldalpha_p^g) = M_g(\F_{p^2})
\]
is surjective, 
so $\Aut B \rightarrow \GL_g(\F_{p^2})$ is surjective.
The latter group acts transitively on the copies of $\boldalpha_p$ in $\boldalpha_p^g$
over $\F_{p^2}$.
\end{proof}

\begin{corollary}
\label{C:quotient is still product}
 Let $B$ be a product $E_1 \times \cdots \times E_g$ of maximal elliptic curves over $\F_{p^2}$.
Let $H$ be a subgroup scheme of $B$ such that $H \isom \boldalpha_p$
or $\#H$ is a prime $\ell \ne p$.
Then $B/H$ is a product of maximal elliptic curves over $\F_{p^2}$.
\end{corollary}

\begin{proof}
 By Lemma~\ref{L:transitive}, we may assume that $H$ is contained is 
the copy of $\boldalpha_p$ in $E_1$, or a cyclic subgroup of order $\ell$ contained in $E_1$.
Then $E_1/H$ is another maximal elliptic curve over $\F_{p^2}$,
and $B/H \isom (E_1/H) \times E_2 \times \cdots \times E_g$.
\end{proof}

\begin{proof}[Proof of Theorem~\ref{T:category of maximal abelian varieties}]
\hfill
\begin{enumerate}[\upshape (a)]
\item 

Among all isogenies from a product of maximal elliptic curves to $A$,
let $\phi \colon B \rightarrow A$ be one of minimal degree
(at least one such $\phi$ exists, 
by Proposition~\ref{P:characterization of maximality}\eqref{I:A is maximal}$\Rightarrow$\eqref{I:A and E^g}).
Let $G$ be the connected component of $\ker \phi$.

 Suppose that $\phi$ is inseparable.
Then $G \ne 0$.
By Lemma~\ref{L:supersingular isogeny factorization},
$G$ contains a copy of $\boldalpha_p$.
By Corollary~\ref{C:quotient is still product},
$B/\boldalpha_p$ is again a product of maximal elliptic curves.
Now $\phi$ factors as $B \rightarrow B/\boldalpha_p \rightarrow A$,
and $B/\boldalpha_p \rightarrow A$ contradicts the minimality of $\phi$.

 Similarly, if $\phi$ is separable and $\deg \phi>1$,
then $\ker \phi$ contains a subgroup $H$ of order $\ell$,
defined over $\F_{p^2}$ by Lemma~\ref{L:defined over F_{p^2}};
Corollary~\ref{C:quotient is still product}
shows again that $B/H \rightarrow A$ contradicts the minimality of $\phi$.

 Hence $\phi$ is an isomorphism,
so $A$ is a product of maximal elliptic curves.
\item 
First let us justify the rewriting of the categories.
F.p.\ torsion-free left $\calO$-modules are projective by
Theorem~\ref{T:modules over quaternion order}\eqref{I:torsion-free equals projective over O}.
By Proposition~\ref{P:characterization of maximality}\eqref{I:A is maximal}$\Leftrightarrow$\eqref{I:A and E^g}, 
the abelian varieties isogenous to a power of $E$
are exactly the maximal abelian varieties over $\F_{p^2}$.

 By~\eqref{I:maximal is product}, 
every maximal abelian variety is a product of maximal elliptic curves,
each of which is $\HOM_{\calO}(I,E)$ for some left $\calO$-ideal $I$,
by the bottom of page~541 in \cite{Waterhouse1969}.
The result now follows from~\ref{T:maximal order equivalence}.
\item 
 Combine Theorem~\ref{T:modules over quaternion order}\eqref{I:free over quaternion order}
and part~\eqref{I:equivalence of categories of maximal abelian varieties}.
\end{enumerate}
The same proofs apply in the minimal case.
\end{proof}

\begin{remark}
 Because of Lemma~\ref{L:homomorphisms defined over Fp^2},
Theorem~\ref{T:category of maximal abelian varieties}\eqref{I:isomorphic to E^g} 
could be deduced also from its analogue over $\Fbar_p$,
that for $g \ge 2$, any product of $g$ supersingular elliptic curves over $\Fbar_p$
is isomorphic to any other.
The latter is a well-known theorem of Deligne, proved in a similar way:
see \cite[Theorem~6.2]{Ogus1979} and~\cite[Theorem~3.5]{Shioda1979}.
\end{remark}

\begin{remark}
 A related result can be found in \cite{Oort1975}: 
Theorem~2 there states that if $A$ is an abelian variety over an algebraically closed field
of characteristic~$p$, and the $a$-number of $A$ equals $\dim A$,
then $A$ is isomorphic to a product of supersingular elliptic curves.
\end{remark}

\section{Kernel subgroups}
\label{S:kernel subgroups}

\subsection{General properties of kernel subgroups}
\label{S:general properties of kernel subgroups}

\begin{definition}
 Let $A$ be an abelian variety over a field.
Call a subgroup scheme $G \subseteq A$ a \emph{kernel subgroup}
if $G = A[I]$ for some $I \subseteq \End A$.
$($These are called \emph{ideal subgroups} in {\rm \cite[p.~302]{Kani2011}}.$)$
\end{definition}

 In the definition, we may replace $I$ by the left $(\End A)$-ideal it generates
without changing $A[I]$.
Thus we may always assume that $I$ is a left \hbox{$(\End A)$-ideal}.

\begin{proposition}
\label{P:kernel subgroups}
\hfill
\begin{enumerate}[\upshape (a)]
\item\label{I:intersection of kernel subgroups}
 An intersection of kernel subgroups in $A$ is a kernel subgroup.
\item \label{I:product of kernel subgroups}
 Let $A_1,\ldots,A_n$ be abelian varieties.
Suppose that $G_i \subseteq A_i$ for $i=1,\ldots,n$.
Then $\prod_{i=1}^n G_i$ is a kernel subgroup of $\prod_{i=1}^n A_i$
if and only if each $G_i$ is a kernel subgroup of $A_i$.
\item \label{I:sum of kernel subgroups}
 Let $I_1,\ldots,I_n$ be pairwise coprime $2$-sided ideals of $\End A$.
Let $G_i \subseteq A[I_i]$ for $i=1,\ldots,n$.
Then $\sum_{i=1}^n G_i$ is a kernel subgroup
if and only if each $G_i$ is a kernel subgroup.
\end{enumerate}
\end{proposition}

\begin{proof}
\hfill
\begin{enumerate}[\upshape (a)]
\item 
 We have $\Intersection A[I_i] = A[\sum I_i]$ for any left ideals $I_i$.
\item 
 Let $A=\prod A_i$ and $G=\prod G_i$.
Suppose that $G=A[I]$.
For each $f \in I$,
the composition $A_i \injects A \stackrel{f}\rightarrow A \surjects A_i$
defines $\bar{f} \in \End A_i$,
and $G_i$ is the intersection of the kernels of all such $\bar{f}$.

 Conversely, suppose that $G_i=A_i[I_i]$ for each $i$.
Let $I\colonequals \prod I_i$ denote the set of ``diagonal'' endomorphisms 
$(f_1,\ldots,f_n) \colon A \rightarrow A$ with $f_i \in I_i$.
Then $G=A[I]$.
\item 
 By induction, we may assume $n=2$.
Since $I_1$ and $I_2$ are coprime $2$-sided ideals,
$A[I_1 I_2] = A[I_1] \directsum A[I_2]$
and every subgroup scheme $H \subseteq A[I_1 I_2]$ decomposes as $H_1 \directsum H_2$
where $H_i \subseteq A[I_i]$;
namely $H_i = H \intersect A[I_i]$.

 If $G_1+G_2$ is a kernel subgroup, 
then so is $G_i = (G_1+G_2) \cap A[I_i]$,
by~\eqref{I:intersection of kernel subgroups}.

 Conversely, suppose that $G_i = A[J_i]$ for some left ideal $J_i$.
Replace $J_i$ by $J_i+I_i$ to assume that $J_i \supseteq I_i$.
Let $K\colonequals I_2 J_1 + I_1 J_2$
We claim that $A[K] = G_1 + G_2$.
First, $I_2 J_1 \subseteq J_1$, which kills $G_1$;
also, $I_1 J_2 \subseteq I_1$, which kills $G_1$.
Thus $K$ kills $G_1$.
Similarly, $K$ kills $G_2$.
Thus $G_1 + G_2 \subseteq A[K]$.
On the other hand, if we write $A[K] = H_1 \directsum H_2$
with $H_i \subseteq A[I_i]$, we will show that $H_i \subseteq G_i$,
so that $A[K] \subseteq G_1+G_2$.
Write $1=e_1+e_2$ with $e_i \in I_i$.
Then the subsets $e_1 J_1 \subseteq I_2 J_1 \subseteq K$ 
and $e_2 J_1 \subseteq I_2$ kill $H_1$,
so $J_1$ kills $H_1$;
i.e., $H_1 \subseteq A[J_1] = G_1$.
Similarly $H_2 \subseteq G_2$.
So $A[K] \subseteq G_1 + G_2$.
Hence $G_1 + G_2 = A[K]$, a kernel subgroup.
\qedhere
\end{enumerate}
\end{proof}

\subsection{Kernel subgroups of a power of an elliptic curve}
\label{S:kernel subgroups of a power}

\begin{proposition}
\label{P:kernel subgroups and quotients}
 Let $E$ be an elliptic curve over a field, and let $r \in \Z_{\ge 0}$.
Let $R \colonequals \End E$.
For a subgroup scheme $G \subseteq E^r$, the following are equivalent\textup{:}
\begin{enumerate}[\upshape (i)]
\item
$G$ is a kernel subgroup.
\item 
$G$ is the kernel of a homomorphism $E^r \rightarrow E^s$ for some $s \in \Z_{\ge 0}$.
\item 
There exists a f.p.\ torsion-free $R$-module $M$
such that $E^r/G \isom \HOM_R(M,E)$.
\item
There exists a submodule $M \subseteq R^r$ such that applying $\HOM_R(-,E)$ to
\[
	0 \rightarrow M \rightarrow R^r \rightarrow R^r/M \rightarrow 0
\]
yields 
\[
	0 \rightarrow G \rightarrow E^r \rightarrow E^r/G \rightarrow 0.
\]
\end{enumerate}
\end{proposition}

\begin{proof}\hfill

 (i)$\Rightarrow$(ii): 
Suppose that $G$ is a kernel subgroup, say $A[I]$.
Let $f_1,\ldots,f_n$ be generators for $I$.
Then $G$ is the kernel of $\xymatrix{ E^r \ar[rr]^-{(f_1,\ldots,f_n)} &&  (E^r)^n}$.

 (ii)$\Rightarrow$(iii): 
This is a special case of Theorem~\ref{T:functor for E}\eqref{I:image of E^r --> E^s}.

 (iii)$\Rightarrow$(iv): 
If $E^r/G \isom \HOM_R(M,E)$ for some f.p.\ torsion-free $M$,
then by Theorem~\ref{T:maximal order equivalence}\eqref{I:fully faithful},
the natural surjection $E^r \surjects E^r/G$
comes from some injection $M \injects R^r$.
Applying $\HOM_R(-,E)$ to 
\[
	0 \rightarrow M \rightarrow R^r \rightarrow R^r/M \rightarrow 0
\]
yields 
\[
	0 \rightarrow H \rightarrow E^r \surjects E^r/G
\]
for some $H$, which must be isomorphic to $G$.

 (iv)$\Rightarrow$(ii):
Choose a surjection $h \colon R^s \surjects M$.
Applying $\HOM_R(-,E)$ to the composition $R^s \stackrel{h}\surjects M \injects R^r$
produces a homomorphism $E^r \surjects E^r/G \injects E^s$ 
with kernel $G$.

 (ii)$\Rightarrow$(i): 
We may increase $s$ to assume that $r|s$.
Then $G$ is an intersection of $s/r$ endomorphisms of $E^r$,
so it is a kernel subgroup 
by Proposition~\ref{P:kernel subgroups}\eqref{I:intersection of kernel subgroups}.
\end{proof}

\begin{proposition}
\label{P:kernel subgroups and equivalence of categories}
 Let $E$ be an elliptic curve.
Let $R\colonequals \End E$.
Then the following are equivalent\textup{:}
\begin{enumerate}[\upshape (i)]
\item
For each $r \in \Z_{\ge 0}$, every subgroup scheme of $E^r$ is a kernel subgroup.
\item 
For each $r \in \Z_{\ge 0}$, every \emph{finite} subgroup scheme of $E^r$ is a kernel subgroup.
\item 
The functors $\HOM_R(-,E)$ and $\Hom(-,E)$
are inverse equivalences of categories.
\end{enumerate}
\end{proposition}

\begin{proof}\hfill

 (i)$\Rightarrow$(ii): Trivial.

 (ii)$\Rightarrow$(iii): 
Suppose that $A$ is an abelian variety isogenous to $E^r$.
Then $A \isom E^r/G$ for some finite subgroup scheme $G$.
By assumption, $G$ is a kernel subgroup.
Proposition~\ref{P:kernel subgroups and quotients}(i)$\Rightarrow$(iii)
implies that $A$ is in the image of $\HOM_R(-,E)$.
The result now follows from Theorem~\ref{T:maximal order equivalence}.

 (iii)$\Rightarrow$(i): 
Let $G$ be a subgroup scheme of $E^r$.
Then $E^r/G$ is isogenous to $E^s$ for some $s \le r$.
By assumption, $\HOM_R(-,E)$ is an equivalence of categories,
so $E^r/G$ is of the form $\HOM_R(M,E)$.
By Proposition~\ref{P:kernel subgroups and quotients}(iii)$\Rightarrow$(i),
$G$ is a kernel subgroup.
\end{proof}

 In the next few sections,
we investigate when it holds that 
all finite subgroup schemes of powers of $E$ are kernel subgroups,
in order to determine when
$\HOM_R(-,E)$ and $\Hom(-,E)$
are inverse equivalences of categories.

\subsection{Prime-to-\texorpdfstring{\protect{\boldmath{$p$}}}{p} subgroups}
\label{S:prime-to-p subgroups}

 We continue to assume that $E$ is an elliptic curve and $R=\End E$.
Let $\ell$ be a prime not equal to $\Char k$.
Let $R_\ell \colonequals R\tensor \Z_\ell$.
The natural map $R_\ell \rightarrow \End_{\Z_\ell} T_\ell E$
is injective since an endomorphism that kills $E[\ell^n]$ for all $n$ is $0$,
and has saturated image since an endomorphism that kills $E[\ell]$ 
is equal to $\ell$ times an endomorphism.
Let $C \colonequals \End_{R_\ell} T_\ell E$, which is the commutant of $R_\ell$ 
in $\End_{\Z_\ell} T_\ell E \isom \M_2(\Z_\ell)$.
For any elliptic curve, we have $\rk R \in \{1,2,4\}$,
so one of the following holds:
\begin{enumerate}[\upshape (i)]
\item  $R_\ell = \Z_\ell$ and $C = \M_2(\Z_\ell)$;
\item\label{I:rank 2} $R_\ell = C = \Z_\ell \directsum \Z_\ell \alpha$,
a $\Z_\ell$-algebra that is a saturated rank~$2$ $\Z_\ell$-submodule of $\M_2(\Z_\ell)$ 
for some $\alpha \in \M_2(\Z_\ell)$; or 
\item  $R_\ell = \M_2(\Z_\ell)$ and $C = \Z_\ell$.
\end{enumerate}
(To see that $C=R_\ell$ in case~\eqref{I:rank 2}, 
one may argue as follows.
By \cite[Corollary~3 in III.\S19]{Mumford1970},
the $\Q$-algebra $R \tensor \Q$ is semisimple,
so $R \tensor \Q_\ell$ is either a degree~$2$ field extension of $\Q_\ell$,
or is conjugate to $\Q_\ell \times \Q_\ell$.
In either case, the commutant $C \tensor_{\Z_\ell} \Q_\ell$ of $R \tensor \Q_\ell$
in $\M_2(\Q_\ell)$ is $2$-dimensional.
On the other hand, an algebra generated by one element is commutative,
so $C$ contains $R_\ell$.
Also, $R_\ell$ is saturated in $\M_2(\Z_\ell)$.
The previous three sentences imply that $C=R_\ell$.)

Let $e \in \Z_{>0}$.

\begin{lemma}
\label{L:C/ell^e C module injects}
 Every finitely generated left $C/\ell^e C$-module injects into a free $C/\ell^e C$-module.
\end{lemma}

\begin{proof}
 If $C=\Z_\ell$, this is trivial.
For any ring $A$ and positive integer $n$, 
the category of $A$-modules is equivalent
to the category of $\M_n(A)$-modules \cite[Theorem~17.20]{Lam1999},
and the equivalence preserves injections, finite generation, 
and projectivity \cite[Remark~17.23(A)]{Lam1999};
applying this to $A=\Z_\ell/\ell^e \Z_\ell$ and $n=2$
shows that the case $C=\Z_\ell$ implies the case $C=\M_2(\Z_\ell)$.

 Finally, suppose that $C$ is of rank~$2$.
Then $C/\ell^e C$ is free of rank~$2$ over $\Z/\ell^e\Z$;
say with basis $1$, $\alpha$.
For $c \in C/\ell^e C$,
let $\lambda(c)$ be the coefficient of $\alpha$ in $c$.
Multiplying any nonzero element of $\ker \lambda$ by $\alpha$ gives an element
outside $\ker \lambda$.
Therefore the pairing 
\begin{align*}
	C/\ell^e C \times C/\ell^e C &\To \Z/\ell^e\Z \\
	x,y &\longmapsto \lambda(xy)
\end{align*}
is a perfect pairing.
In other words, the Pontryagin dual $(C/\ell^e C)^D$ is isomorphic to $C/\ell^e C$
as a $C/\ell^e C$-module.
If $M$ is a finitely generated $C/\ell^e C$-module,
there exists a surjection $(C/\ell^e C)^r \surjects M^D$ for some $r \in \Z_{\ge 0}$;
taking Pontryagin duals yields an injection $M \injects ((C/\ell^e C)^r)^D \isom (C/\ell^e C)^r$.
\end{proof}

\begin{lemma}
\label{L:Tate module is free}
 The group $(T_\ell E)^2$ is free as an $R_\ell$-module and as a $C$-module.
The group $E[\ell^e](k_s)^2$ is free as an $R/\ell^e R$-module and as a $C/\ell^e C$-module.
\end{lemma}

\begin{proof}
 Since $E[\ell^e](k_s) = T_\ell E/\ell^e T_\ell E$,
by Nakayama's lemma 
it is enough to check that $E[\ell](k_s)^2$ is free as an $A$-module,
for $A=R/\ell R$ and for $A=C/\ell C$.
Identify $E[\ell](k_s)^2$ with $\F_\ell^2$,
so that $A \subseteq \M_2(\F_\ell)$.
The case $A=\F_\ell$ is trivial.
If $A$ is $\F_\ell \directsum \F_\ell \alpha$ for some $\alpha \in \M_2(\F_\ell)$,
then every faithful $A$-module of dimension~$2$ over $\F_\ell$ is free.
If $A = \M_2(\F_\ell)$, then the free $A$-module $A$ is 
a direct sum of two copies of $\F_\ell^2$ (the two column spaces).
\end{proof}

\begin{lemma}
\label{L:commutants}
 The natural maps 
\begin{align*}
 	C/\ell^e C &\rightarrow \End_{R/\ell^e R} E[\ell^e](k_s) \\
 	R/\ell^e R &\rightarrow \End_{C/\ell^e C} E[\ell^e](k_s)
\end{align*}
are isomorphisms.
\end{lemma}

\begin{proof}
 The first map is an isomorphism since $C = \End_{R_\ell} T_\ell E$
and $C$ and $R_\ell$ are saturated in $\End T_\ell E \isom \M_2(\Z_\ell)$.
Lemma~\ref{L:Tate module is free} 
and \cite[Theorem~18.8(3)$\Rightarrow$(1)]{Lam1999} imply that
$E[\ell^e](k_s)$ is a generator of 
the category of finitely generated $R/\ell^e R$-modules, 
so \cite[Proposition~18.17(2)(d)]{Lam1999} yields the second isomorphism.
\end{proof}

 Recall that $\Galois_k = \Gal(k_s/k)$.
There is a group homomorphism $\Galois_k \rightarrow C^\times$
since each $\sigma \in \Galois_k$ respects the $R$-action on 
the groups $E[\ell^e](k_s)$ and $T_\ell E$.

\begin{proposition}
\label{P:C/ell^e C-submodule is kernel subgroup}
 Let $E$ be an elliptic curve over a field $k$.
Let $\ell$, $e$, $C$ be as above.
Let $G$ be a subgroup scheme of $E[\ell^e]^r$ for some $r$.
Then $G$ is a kernel subgroup
if and only if $G(k_s)$ is a $C/\ell^e C$-submodule of $E[\ell^e]^r(k_s)$.
\end{proposition}

\begin{proof}
 Suppose that $G(k_s)$ is a $C/\ell^e C$-submodule of $E[\ell^e]^r(k_s)$.
Let $H \colonequals E[\ell^e]^r/G$.
Then $H(k_s)$ is a finitely generated $C/\ell^e C$-module.
By Lemma~\ref{L:C/ell^e C module injects},
$H(k_s)$ injects into a free $C/\ell^e C$-module,
which in turn injects into $E[\ell^e]^s(k_s)$ for some $s$.
Because of the homomorphisms $\Galois_k \rightarrow C^\times \rightarrow (C/\ell^e C)^\times$,
the $C/\ell^e C$-module homomorphism $H(k_s) \rightarrow E[\ell^e]^s(k_s)$
is a $\Galois_k$-module homomorphism,
so it comes from a homomorphism $H \rightarrow E[\ell^e]^s$ of \'etale group schemes.
The composition $E[\ell^e]^r \surjects H \injects E[\ell^e]^s$ 
is given by an $s \times r$ matrix $N_e$ with entries in
$\End_{C/\ell^e C} E[\ell^e](k_s) = R/\ell^e R$ (the equality is Lemma~\ref{L:commutants}).
Lift $N_e$ to $N \in \M_{s \times r}(R)$.
Then $G$ is the intersection of the kernel subgroups $E[\ell^e]^r$
and $\ker(N \colon E^r \rightarrow E^s)$.
By Propositions \ref{P:kernel subgroups}\eqref{I:intersection of kernel subgroups} 
and \ref{P:kernel subgroups and quotients}
$G$ is a kernel subgroup.

 Conversely, if $G$ is a kernel subgroup, say the kernel of $E^r \rightarrow E^s$,
then it is also the kernel of $E[\ell^e]^r \rightarrow E[\ell^e]^s$,
which is a homomorphism of $C/\ell^e C$-modules,
so $G$ is a $C/\ell^e C$-module.
\end{proof}

 The group homomorphism $\Galois_k \rightarrow C^\times$
induces algebra homomorphisms $\Z_\ell[\Galois_k] \rightarrow C$
and $\F_\ell[\Galois_k] \rightarrow C/\ell C$.

\begin{proposition}
\label{P:l-power subgroups}
 Let $E$, $k$, $\ell$, $R$, $R_\ell$, $C$ be as above.
The following are equivalent\textup{:}
\begin{enumerate}[\upshape (i)]
\item
 The homomorphism $\F_\ell[\Galois_k] \rightarrow C/\ell C$ is surjective.
\item
 The homomorphism $\Z_\ell[\Galois_k] \rightarrow C$ is surjective.
\item 
 Every $\ell$-power order subgroup scheme of $E^r$ for every $r$ is a kernel subgroup.
\end{enumerate}
\end{proposition}

\begin{proof}\hfill

 (i)$\Rightarrow$(ii): Nakayama's lemma.

 (ii)$\Rightarrow$(iii): 
Let $G$ be an $\ell$-power subgroup scheme of $E^r$,
say $G \subseteq E[\ell^e]^r$.
Then $G(k_s)$ is a $\Z_\ell[\Galois_k]$-module.
Since $\Z_\ell[\Galois_k] \rightarrow C$ is surjective,
$G(k_s)$ is also a $C$-module,
and hence a $C/\ell^e C$-module.
By Proposition~\ref{P:C/ell^e C-submodule is kernel subgroup},
$G$ is a kernel subgroup.

 (iii)$\Rightarrow$(i): Suppose that $\F_\ell[\Galois_k] \rightarrow C/\ell C$ is not surjective;
let $D$ be the image.
The algebra $C/\ell C$ is one of 
$\F_\ell$, 
$\left\{ \left(\begin{smallmatrix} a & b \\ 0 & a \end{smallmatrix} \right)\right\} \isom \F_\ell[\epsilon]/(\epsilon^2)$,
$\left\{\left( \begin{smallmatrix} a & 0 \\ 0 & b \end{smallmatrix} \right)\right\} \isom \F_\ell \times \F_\ell$,
$\F_{\ell^2}$,
or $\M_2(\F_\ell)$.
The first is excluded since it has no nontrivial subalgebras.
In the second, third, and fourth cases, $D$ can only be $\F_\ell$,
and it is easy to find a subspace of $\F_\ell^2 \isom E[\ell](k_s)$
that is not a $C/\ell C$-module.
In the fifth case, $D$ is contained in a copy
of either $\left\{\left( \begin{smallmatrix} a & b \\ 0 & c \end{smallmatrix}\right) \right\}$
or $\F_{\ell^2}$.
Now $\left\{\left( \begin{smallmatrix} a & b \\ 0 & c \end{smallmatrix}\right) \right\}$ 
fixes a line in $\F_\ell^2$ not fixed by $\M_2(\F_\ell)$.
And $\F_{\ell^2}$ fixes an $\F_{\ell^2}$-line in $\F_{\ell^2}^2 \isom E^2[\ell](k_s)$
that is not fixed by $\M_2(\F_\ell)$.
Thus in each case, there is a subgroup scheme of $E[\ell]$ or $E^2[\ell]$
that is not a $C/\ell C$-module, 
and hence by Proposition~\ref{P:C/ell^e C-submodule is kernel subgroup}
not a kernel subgroup.
\end{proof}

\subsection{\texorpdfstring{\protect{\boldmath{$p$}}}{p}-power subgroups}
\label{S:p-power subgroups}


\begin{proposition}
\label{P:p-power in ordinary}
 Let $E$ be an ordinary elliptic curve over a field $k$ of characteristic~$p$.
Assume that $\End E \ne \Z$ \textup{(}automatic if $k$ is finite\textup{)}.
Then every $p$-power order subgroup scheme $G \subseteq E^r$ is a kernel subgroup.
\end{proposition}

\begin{proof}
 The ring $R\colonequals \End E \isom \End E_{\kbar}$ is a quadratic order.
Although $R$ is not necessarily a Dedekind domain, its conductor is prime to $p$,
so it makes sense to speak of the splitting behavior of $(p)$ in $R$.
In fact, since $E$ is ordinary, $(p)$ splits, say as $\pp \qq$.
So $E[p]$ is the direct sum of group schemes $E[\pp]$ and $E[\qq]$, 
each of order~$p$ by Theorem~\ref{T:functor for E}\eqref{I:order of Hom(N,E)}.
Since $E$ is ordinary, one of them, say $E[\pp]$, is \'etale,
and the other is connected.
For any $e \in \Z_{\ge 0}$, we have $(p^e) = \pp^n \qq^n$
so $E[p^e] \isom E[\pp^e] \directsum E[\qq^e]$.
The Jordan--H\"older factors of $E[\pp^e]$ are isomorphic to $E[\pp]$,
so $E[\pp^e]$ is \'etale; similarly $E[\qq^e]$ is connected.
We have $G \subseteq E[p^e]^r$ for some $e$.
By Proposition~\ref{P:kernel subgroups}\eqref{I:sum of kernel subgroups},
we may assume that $G \subseteq E[\pp^e]^r$
or $G \subseteq E[\qq^e]^r$.

 In the first case, $E[\pp^e](k^s) \isom \Z/p^e\Z$,
so $G$ is the kernel of a homomorphism $E[\pp^e]^r \rightarrow E[\pp^e]^s$
given by a matrix in $\M_{s \times r}(\Z)$.
Since $E[\pp^e]$ is a kernel subgroup, so is $E[\pp^e]^r$, and so is $G$,
by Propositions \ref{P:kernel subgroups}\eqref{I:intersection of kernel subgroups} 
and \ref{P:kernel subgroups and quotients}.

 In the second case, we take Cartier duals:
$E[\pp^e]^r \surjects G^\vee$.
Then $G^\vee$ is the cokernel of some homomorphism $E[\pp^e]^s \rightarrow E[\pp^e]^r$
given by a matrix $N \in \M_{r \times s}(\Z)$.
So $G$ is the kernel of the homomorphism $E[\qq^e]^r \rightarrow E[\qq^e]^s$
given by the transpose $N^T \in \M_{s \times r}(\Z)$.
Since $E[\qq^e]$ is a kernel subgroup, so is $E[\qq^e]^r$,
and so is $G$, 
by Propositions \ref{P:kernel subgroups}\eqref{I:intersection of kernel subgroups} 
and \ref{P:kernel subgroups and quotients}.
\end{proof}

\begin{proposition}
\label{P:p-power in supersingular}
 Let $E$ be a supersingular elliptic curve over a field $k$ of characteristic~$p$.
\begin{enumerate}[\upshape (a)]
\item \label{I:supersingular over F_p}
If $k = \F_p$, then every $p$-power order subgroup scheme $G \subseteq E^r$ is a kernel subgroup,
and in fact is a kernel of an endomorphism of $E^r$.
\item 
If $k = \F_{p^2}$ and $\rk R = 4$ (i.e., $\#E(\F_{p^2}) = (p \pm 1)^2$),
then \emph{every} subgroup scheme $G \subseteq E^r$ is a kernel subgroup.
\item \label{I:alpha_p over F_{p^2}}
If $k = \F_{p^2}$ and $\rk R \ne 4$, 
then there exists a copy of $\boldalpha_p$ in $E \times E$
that is not a kernel subgroup.
\item  \label{I:alpha_p over F_{p^3}}
If $k$ is $\F_{p^a}$ for some $a \ge 3$,
or if $k$ is infinite,
then there exists a copy of $\boldalpha_p$ in $E \times E$
that is not a kernel subgroup.
\end{enumerate}
\end{proposition}

\begin{proof}
 The kernel of $\pi_{E,p} \colon E \rightarrow E^{(p)}$ is $\boldalpha_p$.
Suppose that $\boldalpha_p \subseteq E$ is a kernel subgroup.
By Proposition~\ref{P:kernel subgroups and quotients}(i)$\Rightarrow$(iv), 
$\boldalpha_p \isom \HOM_R(R/I,E)$ for some left \hbox{$R$-ideal} $I$.
By Theorem~\ref{T:functor for E}\eqref{I:order of Hom(N,E)},
$p = \#(R/I)^{2/\rk R}$.
We have three cases:
\begin{itemize}
\item If $R=\Z$, this is a contradiction.
\item If $\rk R = 2$, then $\#(R/I)=p$, so $R/I \isom \F_p$.
Since $E$ is supersingular, $p$ is ramified or inert in $R$,
and the above implies that $p$ is ramified.
\item If $\rk R = 4$, then $\#(R/I)=p^2$, so $I$ is the unique ideal of index $p^2$ in $R$,
and $R/I \isom \F_{p^2}$.
\end{itemize}
If $J$ is an $R$-module with $I^2 \subsetneq J \subsetneq R^2$ (here $I^2$ means $I \times I$),
then $R^2/J \isom R/I$ (since $R/I$ is a field),
and the surjection $R^2 \surjects R^2/J$ gives rise to 
an injection $\boldalpha_p \injects E \times E$.
Conversely, any \emph{kernel} subgroup $\boldalpha_p \subseteq E \times E$
arises from such a $J$.
So such kernel subgroups are in bijection with $\PP^1(R/I)$.
On the other hand, $\End \boldalpha_p \isom k$,
so $\Hom(\boldalpha_p,E \times E) = k^2$,
and the copies of $\boldalpha_p$ in $E \times E$
are in bijection with $\PP^1(k)$.
Thus if every $\boldalpha_p$ in $E \times E$ is a kernel subgroup,
then $\PP^1(R/I)$ is in bijection with $\PP^1(k)$,
so $\#(R/I)=\#k$;
i.e., $k \isom R/I$, which is $\F_p$ or $\F_{p^2}$ as above.
This proves \eqref{I:alpha_p over F_{p^2}} and~\eqref{I:alpha_p over F_{p^3}}.

\begin{enumerate}[\upshape (a)]
\item 
 By Lemma~\ref{L:supersingular isogeny factorization},
$E^r \rightarrow E^r/G$ factors as a chain of $p$-isogenies, each with kernel $\boldalpha_p$.
If we show that any quotient $E^r/\boldalpha_p$ is isomorphic to $E^r$,
then each abelian variety in the chain must be isomorphic to $E^r$,
so $G$ is a kernel of an endomorphism of $E^r$, as desired.

 The group $\GL_r(\Z) \subseteq \GL_r(\End E)$ acts on $E^r$,
and acts transitively on the nonzero elements of $\Hom(\boldalpha_p,E^r)=\F_p^r$.
Therefore it suffices to consider the quotient $E^r/\boldalpha_p$
in which the $\boldalpha_p$ is contained in $E \times 0 \times \cdots \times 0$.
Now $E/\boldalpha_p = E/E[\pi_E] \isom E$,
so $E^r/\boldalpha_p \isom E^r$.
\item 
 The abelian variety $E^r/G$ is isogenous to a power of $E$,
so by Theorem~\ref{T:category of maximal abelian varieties}\eqref{I:equivalence of categories of maximal abelian varieties},
it is of the form $\HOM_R(M,E)$.
By Proposition~\ref{P:kernel subgroups and quotients}(iii)$\Rightarrow$(i),
$G$ is a kernel subgroup.
\qedhere
\end{enumerate}
\end{proof}

\section{Abelian varieties isogenous to a power of an elliptic curve}
\label{S:abelian varieties}

 Let $E$ be an elliptic curve over $k$.
We break into cases, first according to whether $E$ is ordinary or supersingular,
and next according to $\rk \End E$ and $\#k$.
By convention, elliptic curves over a field of characteristic~$0$ are included
among the ordinary curves.

\subsection{\texorpdfstring{\protect{\boldmath{$E$}}}{E} is ordinary and
  \texorpdfstring{\protect{\boldmath{$\rk \End E = 1$}}}{rk End E = 1}}
\label{S:End E = Z}

\begin{theorem}
\label{T:End E = Z}
 Fix an elliptic curve $E$ over a field $k$ such that $\End E \isom \Z$.
\begin{enumerate}[\upshape (a)]
\item 
The image of $\HOM_R(-,E)$
consists of abelian varieties \emph{isomorphic} to a power of $E$.
\item \label{I:Galois not surjective}
The functors $\HOM_R(-,E)$ and $\Hom(-,E)$
are inverse equivalences of categories
\textup{(}i.e., every abelian variety isogenous to a power of $E$ is isomorphic
to a power of $E$\textup{)}
if and only if 
$\Char k=0$ and for every prime $\ell$ 
the homomorphism $\F_\ell[\Galois_k] \rightarrow \End E[\ell](k_s) \isom \M_2(\F_\ell)$
is surjective.
\end{enumerate}
\end{theorem}

\begin{proof}[Proof of Theorem~\ref{T:End E = Z}]
\hfill
\begin{enumerate}[\upshape (a)]
\item  Every f.p.\ torsion-free $\Z$-module is free.
\item 
 By Proposition~\ref{P:kernel subgroups and equivalence of categories},
$\HOM_R(-,E)$ and $\Hom(-,E)$
are equivalences if and only if every finite subgroup scheme $G$ is a kernel subgroup.
By Proposition~\ref{P:kernel subgroups}\eqref{I:sum of kernel subgroups},
we need only consider $G$ of prime power order.

 If $\Char k = p >0$, then $\#\ker \pi_{E,p} = p$, 
but $\#E[I]$ is a square for every nonzero ideal $I \subseteq \Z$,
so $\ker \pi_{E,p}$ is not a kernel subgroup.
If $\Char k = 0$, then apply Proposition~\ref{P:l-power subgroups}(i)$\Leftrightarrow$(iii)
for every $\ell$.
\qedhere
\end{enumerate}
\end{proof}

\begin{remark}
\label{R:Borel and Cartan}
Surjectivity of 
$\F_\ell[\Galois_k] \rightarrow \M_2(\F_\ell)$
\emph{fails} if and only if
the image $G$ of $\Galois_k \rightarrow \GL_2(\F_\ell)$ 
is contained in a Borel subgroup or a nonsplit Cartan subgroup,
as we now explain.
Let $A$ be the image of $\F_\ell[\Galois_k] \rightarrow \M_2(\F_\ell)$.
View $V \colonequals \F_\ell^2$ as an $A$-module.
If $V$ is reducible, then surjectivity fails and $G$
is contained in a Borel subgroup.
So suppose that $V$ is irreducible.
By Schur's lemma \cite[XVII.1.1]{LangAlgebra}, 
$\End_A V$ is a division algebra $D$.
But $D \subseteq \M_2(\F_\ell)$, so $D$ is $\F_\ell$ or $\F_{\ell^2}$.
By Wedderburn's theorem \cite[XVII.3.5]{LangAlgebra}, $A \isom \End_D V$.
If $D=\F_\ell$, then $A=\M_2(\F_\ell)$, and $G$
is not contained in a Borel subgroup or a nonsplit Cartan subgroup.
If $D \isom \F_{\ell^2}$, then $\dim_D V =1$,
so $A \isom \End_D V \isom \F_{\ell^2}$, 
and $G$ is contained in 
the nonsplit Cartan subgroup $A \intersect \GL_2(\F_\ell)$.
\end{remark}

\begin{example}
Let $E$ be the elliptic curve $X_0(11)$ over $\Q$, 
with equation $y^2+y=x^3-x^2-10x-20$.  As in \cite[5.5.2]{Serre1972}, 
the image of $\Galois_\Q \rightarrow \Aut E[5] \isom \GL_2(\F_5)$ 
is contained in a Borel subgroup, 
so by Theorem~\ref{T:End E = Z}\eqref{I:Galois not surjective}
and Remark~\ref{R:Borel and Cartan},
the functors $\HOM_R(-,E)$ and $\Hom(-,E)$
are \emph{not} inverse equivalences of categories.
\end{example}

\begin{example}
Let $E$ be the elliptic curve over $\Q$ of conductor $37$
with equation $y^2+y=x^3-x$.  
By \cite[5.5.6]{Serre1972},  
the homomorphism $\Galois_\Q \rightarrow \Aut (E[\ell]) \isom \GL_2(\F_{\ell})$ 
is surjective for every prime $\ell$,
so by Theorem~\ref{T:End E = Z}\eqref{I:Galois not surjective}, 
the functors $\HOM_R(-,E)$ and $\Hom(-,E)$
\emph{are} inverse equivalences of categories.
\end{example}

\subsection{\texorpdfstring{\protect{\boldmath{$E$}}}{E} is ordinary and
  \texorpdfstring{\protect{\boldmath{$\rk \End E = 2$}}}{rk End E = 2}} 
\label{S:rank 2}

 Fix an ordinary elliptic curve $E$ over a field $k$
such that $\rk \End E=2$.
(These are called \emph{CM elliptic curves} in \cite[Section~3]{Kani2011}.)
Then $\End E \isom \End E_{\kbar}$, 
because if an endomorphism becomes divisible by a positive integer $n$ over
an extension field, it kills $E[n]$, so it is divisible by $n$ already over $k$.
Let $R\colonequals \End E$ and $K\colonequals \Frac R$.

 If $E'$ is an elliptic curve isogenous to $E$,
then $\End E'$ is another order $R'$ in $K$.
Let $f_{E'}$ be the conductor of $R'$,
i.e., the index of $R'$ in its integral closure.
More generally, if $A$ is an abelian variety isogenous to $E^r$,
then $\End A$ is an order in $M_r(K)$,
and its center $Z(\End A)$ is an order in $Z(M_r(K))=K$,
and we let $f_A$ be the conductor of $Z(\End A)$.

\begin{theorem}
\label{T:equivalence of categories for CM curves}
 Fix an ordinary elliptic curve $E$ over a field $k$ such that $\rk \End E=2$.
Let $R\colonequals \End E$.
The image of $\HOM_R(-,E)$
consists of the abelian varieties $A$ isogenous to a power of $E$
such that $f_A|f_E$, i.e., such that $R \subseteq Z(\End A)$.
These are exactly the products of elliptic curves $E'$
each isogenous to $E$ and satisfying $f_{E'}|f_E$.
\end{theorem}

\begin{proof}
 Suppose that $\phi \colon E^r \rightarrow A$ is an isogeny and $f_A|f_E$.
Since $f_A|f_E$, there is an $R$-action on $A$ 
such that $\phi$ respects the $R$-actions.
Let $G\colonequals \ker \phi$, so $G(k_s)$ is an $R$-module.
Write $G = \Directsum_\ell G_\ell$, where $G_\ell$ is a group scheme of
$\ell$-power order.
For $\ell \ne \Char k$, 
we are in the case $R_\ell=C$ of Section~\ref{S:prime-to-p subgroups},
so $G_\ell(k_s)$ is also a $C/\ell^e C$-module for some $e$, 
and Proposition~\ref{P:C/ell^e C-submodule is kernel subgroup}
shows that $G_\ell$ is a kernel subgroup.
If $\Char k = p>0$, then $G_p$ is a kernel subgroup by Proposition~\ref{P:p-power in ordinary}.
By Proposition~\ref{P:kernel subgroups}\eqref{I:sum of kernel subgroups},
$G$ is a kernel subgroup.
By Proposition~\ref{P:kernel subgroups and quotients}(i)$\Rightarrow$(iii),
the abelian variety $A \isom E^r/G$ is in the image of $\HOM_R(-,E)$.

 Conversely, if $A$ is in the image of $\HOM_R(-,E)$
then by Theorem~\ref{T:maximal order equivalence}\eqref{I:essential image consists of products},
$A$ is a product of elliptic curves of the form $\HOM_R(I,E)$.
Because the functor $\HOM_R(-,E)$ is fully faithful,
if $E' = \HOM_R(I,E)$ then $E'$ is isogenous to $E$
and $\End E' \isom \End_R I$, which contains $R$ since $R$ is commutative.
In particular, $f_{E'} | f_E$.
Finally, $f_A$ is the least common multiple of the $f_{E'}$,
so $f_A | f_E$ too.
\end{proof}

\begin{theorem}
\label{T:ordinary elliptic curve}
 Fix an ordinary elliptic curve $E$ over a finite field $\F_q$.
Let $R\colonequals \End E$.
Then $\HOM_R(-,E)$ and $\Hom(-,E)$
are equivalences of categories
if and only if $\Z[\pi_E]=R$.
\end{theorem}

\begin{proof}
 Suppose that $\Z[\pi_E]=R$.
If $A$ is isogenous to a power of $E$,
then $\pi_A$ has the same minimal polynomial as $\pi_E$,
so $Z(\End A)$ contains $\Z[\pi_A] \isom \Z[\pi_E]$;
i.e., $f_A|f_E$ is automatic.

 On the other hand, if $\Z[\pi_E] \ne R$,
then $E$ is isogenous to an elliptic curve $E'$
satisfying $\End E' = \Z[\pi_{E'}]$ \cite[Theorem~4.2(2)]{Waterhouse1969}.
Theorem~\ref{T:equivalence of categories for CM curves}
shows that $E'$ is not in the image of $\HOM_R(-,E)$,
so $\HOM_R(-,E)$ is not an equivalence of categories.
\end{proof}

We can also give a more general criterion 
that applies even if $k$ is not finite.

\begin{theorem}
\label{T:ordinary over infinite field}
 Fix an ordinary elliptic curve $E$ over a field $k$ such that $\rk \End E=2$.
Then $\HOM_R(-,E)$ and $\Hom(-,E)$
are equivalences of categories
if and only if for every prime $\ell \ne \Char k$,
there exists $\sigma \in \Galois_k$ whose action on $E[\ell](k_s)$
is not multiplication by a scalar.
\end{theorem}

\begin{proof}
 By Propositions \ref{P:kernel subgroups and equivalence of categories}(i)$\Leftrightarrow$(iii), 
\ref{P:kernel subgroups}\eqref{I:sum of kernel subgroups}, 
and~\ref{P:p-power in ordinary},
the functors $\HOM_R(-,E)$ and $\Hom(-,E)$
are equivalences if and only if for each $\ell \ne \Char k$,
the homomorphism $\F_\ell[\Galois_k] \rightarrow C/\ell C$ is surjective.
Since $\dim_{\F_\ell} C/\ell C = 2$, surjectivity is equivalent
to the image of $\F_\ell[\Galois_k] \rightarrow C/\ell C \subseteq \End E[\ell](k_s) \isom \M_2(\F_\ell)$ 
not being $\F_\ell$.
\end{proof}

\begin{example}
Let $E$ be the elliptic curve $y^2=x^3-x$ over
$k \colonequals \Q(\sqrt{-1})$; then $j(E)=1728$ and $\End E =\Z[\sqrt{-1}]$.
The group $\Galois_k$ acts trivially on $E[2](k_s)$,
so by Theorem~\ref{T:ordinary over infinite field}, 
the functors $\HOM_R(-,E)$ and $\Hom(-,E)$
are \emph{not} inverse equivalences of categories.
\end{example}

\begin{example}
Let $E$ be the elliptic curve $y^2=x^3+x^2-3x+1$ 
over $k \colonequals \Q(\sqrt{-2})$;
then $j(E)=8000$ and $\End E =\Z[\sqrt{-2}]$.
The field $k(E[2])$ equals $k(\sqrt{-1})$, so the image of $\Galois_k$
in $\GL_2(\F_2)$ has order $2$ and hence does not consist of scalars.
Now consider a prime $\ell>2$.
Choose a prime $p \ne \ell$ such that $p$ splits in $k/\Q$ and
$\left(\frac{p}{\ell}\right)=-1$.
Let $\sigma$ be a Frobenius element of $\Galois_k$ at a prime above $p$.
The image of $\sigma$ in $\Aut(E[\ell])\simeq \GL_2(\F_\ell)$
has nonsquare determinant $(p \bmod \ell)$,
so it is not a scalar.
Thus, by Theorem~\ref{T:ordinary over infinite field},
the functors $\HOM_R(-,E)$ and $\Hom(-,E)$
\emph{are} inverse equivalences of categories.
\end{example}

\begin{remark}
If $E$ and $E'$ are ordinary elliptic curves over an algebraically closed field $k$
and their endomorphism rings are orders in the same quadratic field,
then $E$ and $E'$ are isogenous.
But over non-algebraically closed fields, this can fail.
For example, if $E$ is an ordinary elliptic curve over a finite field,
then its quadratic twist $E'$ has the same endomorphism ring,
but opposite trace of Frobenius, so $E$ and $E'$ are not isogenous.
\end{remark}

\subsection{\texorpdfstring{\protect{\boldmath{$E$}}}{E} is supersingular and
  \texorpdfstring{\protect{\boldmath{$k=\F_p$}}}{k = Fp}}
\label{S:p+1 points}

 Fix a supersingular elliptic curve $E$ over $\F_p$.
Let $R\colonequals \End E$.
Let $P(x)$ be the characteristic polynomial of $\pi\colonequals \pi_E$.
Define $f_A$ as in Section~\ref{S:rank 2}.
In particular, $f_E$ is the conductor of $R$.
We have the following cases:
\begin{center}
\begin{tabular}{cccc|cc}
prime & $P(x)$ & $\Z[\pi]$ & $R=\End E$ & $f_E$ & equivalence? \\ \hline
$p \not\equiv 3 \pmod{4}$ & $x^2+p$ & $\Z\left[\sqrt{-p}\right]$ & 
$\Z\left[\sqrt{-p}\right]$ & $1$ & YES \\
$p \equiv 3 \pmod{4}$ & $x^2+p$ & $\Z\left[\sqrt{-p}\right]$ & 
$\Z\left[\sqrt{-p}\right]$ & $2$ &YES \\
$p \equiv 3 \pmod{4}$ & $x^2+p$ & $\Z\left[\sqrt{-p}\right]$ & 
$\Z\left[\frac{1+\sqrt{-p}}2\right]$ & $1$ & NO \\
$p=2$ & $x^2 \pm 2x + 2$ & $\Z[i]$ & $\Z[i]$ & $1$ & YES \\
$p=3$ & $x^2 \pm 3x + 3$ & $\Z\left[\frac{-1+\sqrt{-3}}2\right]$ &
$\Z\left[\frac{-1+\sqrt{-3}}2\right]$  & $1$ & YES \\
\end{tabular}
\end{center}
\bigskip

\noindent The last column, which indicates when 
$\HOM_R(-,E)$ and $\Hom(-,E)$
are equivalences of categories,
is explained by the following analogues of 
Theorems \ref{T:equivalence of categories for CM curves} and~\ref{T:ordinary elliptic curve},
proved in the same way except that we use
Proposition~\ref{P:p-power in supersingular}\eqref{I:supersingular over F_p}
in place of Proposition~\ref{P:p-power in ordinary}.

\begin{theorem}
\label{T:equivalence of categories for supersingular over F_p}
 Fix a supersingular elliptic curve $E$ over $\F_p$.
Let $R\colonequals \End E$.
The image of $\HOM_R(-,E)$
consists of the abelian varieties $A$ isogenous to a power of $E$
such that $f_A|f_E$, i.e., such that $R \subseteq Z(\End A)$.
These are exactly the products of elliptic curves $E'$
each isogenous to $E$ and satisfying $f_{E'}|f_E$.
\end{theorem}

\begin{theorem}
\label{T:supersingular E over F_p}
 Fix a supersingular elliptic curve $E$ over $\F_p$.
Let $R\colonequals \End E$.
Then $\HOM_R(-,E)$ and $\Hom(-,E)$
are equivalences of categories
if and only if $\Z[\pi_E]=R$.
\end{theorem}

\subsection{\texorpdfstring{\protect{\boldmath{$E$}}}{E} is supersingular,
  \texorpdfstring{\protect{\boldmath{$k=\F_{p^2}$, and $\rk \End E = 4$}}}{k = Fp2, and rk End E = 4}} 
\label{S:maximal minimal}

 In this case, $E$ is a maximal or minimal elliptic curve over $\F_{p^2}$.
These cases were already handled: see Theorem~\ref{T:category of maximal abelian varieties}.

\subsection{\texorpdfstring{\protect{\boldmath{$E$}}}{E} is supersingular, \texorpdfstring{\protect{\boldmath{$k=\F_{p^2}$, and $\rk \End E = 2$}}}{k = Fp2, and rk End E = 2}}
\label{S:weird supersingular}

 By Proposition~\ref{P:p-power in supersingular}\eqref{I:alpha_p over F_{p^2}},
not every subgroup scheme is a kernel subgroup.
By Proposition~\ref{P:kernel subgroups and equivalence of categories}(iii)$\Leftrightarrow$(ii),
the functors $\HOM_R(-,E)$ and $\Hom(-,E)$
are not equivalences of categories.

\begin{remark}
 These are the cases in which the characteristic polynomial of $\pi_E$ 
is one of $x^2 + px + p^2$, $x^2 + p^2$, or $x^2 - px + p^2$.
Hence $\pi_E=p\zeta$ for a root of unity $\zeta$ of order $3$, $4$, or $6$, respectively.
But $p$ does not divide the conductor of $R$, so $\zeta \in R$.
Now $\zeta \in \Aut E$, so $E$ has $j$-invariant $0$ or $1728$.
\end{remark}

\subsection{\texorpdfstring{\protect{\boldmath{$E$}}}{E} is supersingular and \texorpdfstring{\protect{\boldmath{$\#k > p^2$}}}{\#k > p2}}
\label{S:supersingular over large field}

By Proposition~\ref{P:p-power in supersingular}\eqref{I:alpha_p over F_{p^3}},
not every subgroup scheme is a kernel subgroup.
By Proposition~\ref{P:kernel subgroups and equivalence of categories}(iii)$\Leftrightarrow$(ii),
the functors $\HOM_R(-,E)$ and $\Hom(-,E)$
are not equivalences of categories.

\section{A partial generalization to higher-dimensional abelian varieties over $\F_p$}
\label{S:CS}

Let $B$ be an abelian variety over a prime field $\F_p$.
Let $R \subseteq \End B$ be the (central) subring $\Z[F,V]$
generated by the Frobenius and Verschiebung endomorphisms.
Given a f.p.\ reflexive $R$-module $M$,
let $M^* \colonequals \Hom_R(M,R)$; then $M^*$ is reflexive too.

As in the case of elliptic curves, we can define functors 
\begin{align*}
	\HOM_R(-,B) \colon & \{ \textup{f.p.\ $R$-modules} \}^{\opp} \\
	& \rightarrow
	\{\textup{commutative proper group schemes over $\F_p$}\}
\end{align*}
and 
\begin{align*}
	\Hom(-,B) \colon 
	& \{\textup{commutative proper group schemes over $\F_p$}\} \\
	& \rightarrow
	\{ \textup{f.p.\ $R$-modules} \}^{\opp}.
\end{align*}
The work of Centeleghe and Stix \cite{Centeleghe-Stix2015}, 
combined with some further arguments,
allows us to analyze this higher-dimensional case.
The main extra ingredient we supply is that, under appropriate hypotheses, 
the functor $M \mapsto M^* \tensor_R B$ implicit in \cite{Centeleghe-Stix2015}
is isomorphic to $\HOM_R(-,B)$.

\begin{theorem}
\label{T:powers of B}
Let $B$ be an abelian variety over $\F_p$.
Let $R = \Z[F,V] \subseteq \End B$.
Then the functors $\HOM_R(-,B)$ and $\Hom(-,B)$ 
restrict to inverse equivalences of categories
\[
\xymatrix{
	\{ \textup{f.p.\ reflexive $R$-modules} \}^{\opp} \ar@<0.5ex>[r]
	&
	\{\textup{abelian variety quotients of powers of $B$}\} \ar@<0.5ex>[l]
}
\]
if and only if $R = \End B$.
Moreover, in this case, the functor $\HOM_R(-,B)$ so restricted is exact,
and it is isomorphic to the functor $M \mapsto M^* \tensor B$.
\end{theorem}

\begin{proof}
If the functors give inverse equivalences as stated,
then the argument in the paragraph before 
Theorem~\ref{T:maximal order equivalence}
proves that $R = \End B$.

Now let us prove the converse.
Suppose that $R=\End B$.
Then $(\End B) \tensor \Q$ is commutative.
This implies that in the decomposition of $B$
into simple factors up to isogeny,
no factor is repeated,
and also no factor is associated to the Weil number $\sqrt{p}$,
since such a factor would give a direct factor of $(\End B) \tensor \Q$ 
isomorphic to a quaternion algebra over $\Q(\sqrt{p})$: 
see \cite[p.~528, Case~2]{Waterhouse1969}.
Let $w$ be the set of Weil number conjugacy classes associated to $B$.
Then the category $\AV_w$ of \cite[5.1]{Centeleghe-Stix2015}
is the category of abelian variety quotients of powers of $B$.
The ring $R_w$ in \cite[Definition~2]{Centeleghe-Stix2015} is $R=\Z[F,V]$.
It is Gorenstein by \cite[Theorem~11(2)]{Centeleghe-Stix2015}.
Reflexive finitely generated $R$-modules are the same 
as f.p.\ torsion-free $R$-modules,
or equivalently f.p.\ $R$-modules that are free over $\Z$ 
\cite[Lemma~13]{Centeleghe-Stix2015}.
By \cite[Proposition~24]{Centeleghe-Stix2015}, for every prime $\ell$ 
the $(R \tensor \Z_\ell)$-module $T_\ell B$
(Tate module or contravariant Dieudonn\'e module)
is free of rank~$1$,
so the abelian variety $A_w$ in \cite[Proposition~21]{Centeleghe-Stix2015}
may be taken to be $B$ by \cite[Proposition~24]{Centeleghe-Stix2015}.

We now check that if $M$ is a f.p.\ torsion-free $R$-module,
then the commutative proper group scheme $G \colonequals \HOM_R(M,B)$ 
is an abelian variety.
It suffices to prove that for every prime $\ell$ and $n \ge 0$,
the homomorphism $G[\ell^{n+1}] \stackrel{\ell}\ra G[\ell^n]$ is surjective.
Choose a presentation $R^a \stackrel{N}\ra R^b \ra M \ra 0$,
so $G \colonequals \ker (B^b \ra B^a)$.
Both $M$ and $M^*$ are reflexive $R$-modules,
so they are free over $\Z$.

Suppose that $\ell \ne p$.
Then
\begin{align*}
G[\ell^n] &= (\ker (B^b \ra B^a))[\ell^n] \\
	&= \ker(B^b[\ell^n] \ra B^a[\ell^n]) \\
	&\isom \ker(T_\ell(B^b)/\ell^n \ra T_\ell(B^a)/\ell^n) \\
	&= \ker((R/\ell^n)^b \stackrel{N^T}\ra (R/\ell^n)^a) \\
	&= \ker(\Hom_R(R^b,R/\ell^n) \stackrel{N^T}\ra \Hom_R(R^a,R/\ell^n)) \\
	&\isom \Hom_R(M,R/\ell^n) \\
	& \isom M^*/\ell^n \quad \textup{(since $\Ext^1_R(M,R)=0$ by \cite[Lemma~17]{Centeleghe-Stix2015})}\\
	&= M^* \underset{\Z}\tensor \frac{\ell^{-n}\Z}{\Z}.
\end{align*}
Since $M^*$ is free over $\Z$,
the homomorphism
\[
	M^* \underset{\Z}\tensor \frac{\ell^{-(n+1)}\Z}{\Z} 
	\stackrel{\ell}\longrightarrow
	M^* \underset{\Z}\tensor \frac{\ell^{-n}\Z}{\Z}
\]
is surjective, so $G[\ell^{n+1}] \stackrel{\ell}\ra G[\ell^n]$ is surjective.

Now suppose that $\ell=p$.
For each commutative group scheme $H$ over $\F_p$,
let $H^D$ denote its contravariant Dieudonn\'e module.
Since the $R \tensor \Z_p$-module $T_p B$ is free of rank~$1$,
we have $B[p^n]^D \isom R/p^n$ as an $R$-module.
Next, 
\begin{align*}
G[p^n]^D &= \coker\left( B^a[p^n]^D \ra B^b[p^n]^D \right) \\
	&\isom \coker\left( (R/p^n)^a \stackrel{N}\ra (R/p^n)^b \right) \\
	&= M/p^n.
\end{align*}
Since $M$ is free over $\Z$,
the homomorphism $M/p^n \stackrel{p}\ra M/p^{n+1}$ is injective,
so $G[p^n]^D \stackrel{p}\ra G[p^{n+1}]^D$ is injective,
so $G[p^{n+1}] \stackrel{p}\ra G[p^n]$ is surjective.

Thus $G$ is an abelian variety.
The proof of Theorem~\ref{T:functor for E}\eqref{I:Hom is exact}
now shows that $\HOM_R(-,B)$ is exact.
In particular, if 
$0 \ra M \ra R^n \ra R^m$ is an exact sequence of $R$-modules,
then 
$B^m \ra B^n \ra \HOM_R(M,B) \ra 0$ is exact.
But $M^* \tensor B$ too is defined as $\coker(B^m \ra B^n)$,
so $\HOM_R(M,B) \isom M^* \tensor_R B$,
and this holds functorially in $M$.

Finally, by \cite[Theorem~25 and p.~247]{Centeleghe-Stix2015},
\[
	\Hom(-,B) \colon 
	\AV_w \To \{ \textup{f.p.\ reflexive $R$-modules} \}^{\opp}
\]
is an equivalence of categories with inverse functor
$M \mapsto M^* \tensor_R B$.
We may replace the latter with the isomorphic functor $\HOM_R(-,B)$.
\end{proof}

\begin{remark}
Over $\F_{p^n}$ with $n>1$,
the functors $\HOM_R(-,B)$ and $\Hom(-,B)$ 
are sometimes inverse equivalences of categories,
and sometimes not,
as we saw already in the case of elliptic curves: 
see Theorem~\ref{T:elliptic curve over finite field}.
\end{remark}

\section*{Acknowledgments} 
It is a pleasure to thank Everett Howe, Tony Scholl, and 
Christopher Skinner for helpful discussions. 
We thank also the referees for valuable suggestions on the exposition.

\end{document}